\documentclass[11pt]{amsart}

\usepackage{graphicx, amsmath, amssymb, amsthm, color}
\advance \topmargin by -\headheight
\advance \topmargin by -\headsep
\evensidemargin \oddsidemargin
\newtheorem{thm}{Theorem}[section]
\newtheorem{cor}[thm]{Corollary}
\newtheorem{lem}[thm]{Lemma}
\newtheorem{prop}[thm]{Proposition}

\theoremstyle{definition}

\theoremstyle{remark}
\newtheorem{rem}[thm]{Remark}
\numberwithin{equation}{section}

\newcommand{\leg}[2]{\left(\frac{#1}{#2}\right)}

\newcommand{\Z}{\mathbb Z}
\newcommand{\C}{\mathbb C}
\newcommand{\F}{\mathbb F}
\newcommand{\Q}{\mathbb Q}

\newcommand{\Fha}[4]{{}_2F_1\left(\left. \begin{array}{cc} #1 & #2 \\     \  & #3  \end{array}\right| #4 \right)}
\newcommand{\Fhn}[4]{{}_{n+1}F_n\left(\left. \begin{array}{cc} #1 & #2 \\     \  & #3  \end{array}\right| #4 \right)}
\newcommand{\Fhth}[6]{{}_3F_2\left(\left. \begin{array}{ccc}  #1 & #2 & #3 \\ \  & #4 & #5   \end{array}\right| #6 \right)}
\newcommand{\Fhf}[8]{{}_4F_3\left(\left. \begin{array}{cccc}  #1 & #2 & #3 & #4\\ \  & #5 & #6 & #7   \end{array}\right| #8 \right)}

\newcommand{\tr}{\textrm{tr}}

\newcommand{\ord}{\textrm{ord}}

\newcommand{\case}[1]{\noindent\textbf{Case #1:}}
\newcommand{\hs}[1]{h^*\left (#1 \right )}
\newcommand{\h}[1]{h\left (#1 \right )}

\newcommand{\ep}{\epsilon}

\newcommand{\bin}[2]{\left( {#1 \atop #2} \right)}

\begin{document}
\title[Traces of Hecke Operators and Hypergeometric Functions]{Trace Formulas for Hecke Operators, Gaussian Hypergeometric Functions, and the Modularity of a Threefold}
\author{Catherine Lennon}
\thanks{This work was supported by the Department of Defense (DoD) through the National Defense Science \& Engineering Graduate Fellowship (NDSEG) Program.}

\begin{abstract}
We present here simple trace formulas for  Hecke operators $T_k(p)$ for all $p>3$ on  $S_k(\Gamma_0(3))$ and $S_k(\Gamma_0(9))$, the spaces of cusp forms of weight $k$ and levels 3 and 9.  These formulas can be expressed in terms of special values of Gaussian hypergeometric series and lend themselves to  simple recursive expressions in terms of traces of Hecke operators on spaces of lower weight.  Along the way, we show how to express the traces of Frobenius of a family of elliptic curves with 3-torsion as special values of a Gaussian hypergeometric series over $\F_q$, when $q\equiv 1 \pmod{3}$.  We also use these formulas to provide a simple expression for the Fourier coefficients of $\eta(3z)^8$, the unique normalized cusp form of weight 4 and level 9, and then show that the number of points on a certain threefold is expressible in terms of these coefficients.
\end{abstract}

\maketitle
\section{Introduction and statement of results}
In this paper, we consider the problem of expressing traces of Hecke operators in terms of Gaussian hypergeometric series, where these functions are the finite field analogues of classical hypergeometric series. Although in general the trace formula is quite complicated, recent work such as \cite{ahlg,ahlgo,frech,fuselier} has shown that Gaussian hypergeometric series provide a natural way to express trace formulas and a useful tool for simplifying expressions for the  Fourier coefficients of cusp forms. We continue to study the connections between trace formulas and hypergeometric series in this paper and provide simple recursive formulas for levels 3 and 9. In order to do this, we parametrize elliptic curves with 3-torsion in such a way that is easy to determine how many isogenous curves exist.  We use the results obtained to provide a simple expression for the Fourier coefficients of $\eta(3z)^8$, the unique normalized cusp form of weight 4 and level 9, in terms of Jacobi sums.  Using purely elementary techniques, we are able to then provide a threefold whose number of points over $\F_p$ can be expressed in terms of the Fourier coefficients of a modular form.

As a step along the way, we also prove that the trace of the Frobenius endomorphism on curves in our family is equal to a special value of a Gaussian hypergeometric series. Earlier results such as \cite{koike, ono} have proven formulas  for other classes of elliptic curves, including the Legendre family. As in these cases, the classical hypergeometric series analogue of the Gaussian hypergeometric series  obtained matches that giving the period of the elliptic curve. This is not surprising given the relationship between periods of curves and their Hasse-Witt matrix, as well as the strong connection between Gaussian hypergeometric series and their classical counterparts \cite{clemens, manin, igusa}.

We begin with some preliminary definitions needed to state  our results. Let $q=p^e$ be a power of an odd prime and $\F_q$ the finite field of $q$ elements. Extend each character $\chi \in \widehat{\F^*_q}$ to all of $\F_q$ by setting $\chi(0):=0$.  For any two characters $A,B\in \widehat{\F^*_q}$ one can  define the normalized Jacobi sum by
\begin{equation}
\bin{A}{B}:=\frac{B(-1)}{q}\sum_{x\in \F_q}A(x)\bar{B}(1-x)=\frac{B(-1)}{q} J(A, \bar{B})
\end{equation}
where $J(A,B)$ denotes the usual Jacobi sum.

Recall the definition of the  \emph{Gaussian hypergeometric series over $\F_q$} first defined by Greene in \cite{greene}. For any positive integer $n$ and characters $A_0,A_1,...,A_n$ and $B_1,B_2,...,B_n \in \widehat{\F_q^*}$, the Gaussian hypergeometric series $_{n+1}F_n$ is defined to be
\begin{equation}
 _{n+1}F_n \left( \left.{\begin{array}{cccc}
                A_0 & A_1 & ... & A_n \\
                \  & B_1 & ...  & B_n
              \end{array}
}\right|x\right)_q:=\frac{q}{q-1}\sum_{\chi \in \widehat{\F_q^*}}\left(A_0\chi \atop \chi\right)\left(A_1 \chi\atop B_1 \chi\right)...\left(A_n \chi \atop B_n \chi \right)\chi(x).
\end{equation}
We will sometimes drop the subscript $q$ when it is clear what field we are working in, and just write $\Fhn{A_0}{A_1\ ...\  A_n}{B_1\ ...\  B_n}{x}$. See also Katz \cite{katz} (in particular Section 8.2) for more information on how these sums naturally arise as the traces of Frobenius at closed points of certain $\ell$-adic hypergeometric sheaves.

Gaussian hypergeometric series are of interest because of their connection to the arithmetic properties of varieties, as demonstrated in \cite{koike,ono,fuselier}.  In this paper, we provide further evidence for this connection; in particular, consider an elliptic curve over $\Q$ in the form

\begin{equation}\label{eform} E_{a_1,a_3}:  y^2+ a_1 xy+a_3 y=x^3,
\end{equation}
where $a_i\in \Z$. If  $p$ is a prime for which $E_{a_1,a_3}$ has good reduction, let $\tilde{E}_{a_1,a_3}$ denote the same curve reduced modulo $p$ and $\tilde{E}_{a_1,a_3}(\F_q)$ its $\F_q$-rational points. For each $q=p^e$, write the trace  of the Frobenius map on $\tilde{E}_{a_1,a_3}(\F_q)$ as $t_q(E_{a_1,a_3})$, so that
\begin{equation}
t_q(E_{a_1,a_3})=q+1-|\tilde{E}_{a_1,a_3}(\F_q)|.
\end{equation}
 Then this value may be expressed in terms of Gaussian hypergeometric functions as follows.

\begin{thm}\label{hyper}
Let $E_{a_1,a_3}$ be an elliptic curve over $\Q$ in the form given in equation (\ref{eform}) and let $p$ be a prime for which $E_{a_1,a_3}$ has good reduction. Also assume that $p\nmid a_1$, and $q=p^e\equiv 1 \pmod{3}$.  Let $\rho \in \widehat{\F_q^*}$ be a character of order three, and let $\epsilon$ be the trivial character. Then the trace of the Frobenius map on $\tilde{E}_{a_1,a_3}(\F_q)$ is given by $$t_q(E_{a_1,a_3})=-q\cdot \Fha{\rho}{\rho^2}{\epsilon}{27a_1^{-3}a_3}_q.$$
\end{thm}

 If we write the Hasse-Weil $L$-function of $E_{a_1,a_3}$ as
\begin{equation}
L(s,E_{a_1,a_3})=\sum_{n=1}^{\infty}a_n(E_{a_1,a_3})n^{-s}
\end{equation}
then we can use Theorem \ref{hyper} to express the coefficients $a_n(E_{a_1,a_3})$ in terms of Gaussian hypergeometric functions.
Expressions for the trace of Frobenius like the one above turn out to be a common phenomenon and generalizing this result will be the subject of another paper \cite{lennon}.

For each integer $k\geq 2$ we denote the  space of cusp forms of weight $k$ and trivial character on $\Gamma_0(N)$ by $S_k(\Gamma_0(N))$. For each integer $n$ such that $\gcd(n,N)=1$, let $T_k(n)$ denote the $n$th Hecke operator on this space and $\tr_k(\Gamma_0(N),n)$ the trace of $T_k(n)$. We will prove that $\tr_k(\Gamma_0(3),p)$ can be expressed as follows:

\begin{thm}\label{tracefor3} Let $p\neq 3$ be prime. For $t\in \F_p^*$, let $E_t:=E_{t,t^2}$ denote the curve $y^2+txy+t^2y=x^3$, and let $a_{q}(E_t)$ be the coefficient of $q^{-s}$ in the Hasse-Weil $L$-function of $E_t$. For any even $k\geq 4$, the trace of $T_k(p)$ on $S_k(\Gamma_0(3))$ can be written as
                    $$\tr_k(\Gamma_0(3),p)=-\sum_{t\in \F_p, \Delta(E_t) \neq 0} a_{p^{k-2}}(E_t)-\gamma_k(p)-2,$$
where
\begin{equation}\label{gamma}
\gamma_k(p):=\left\{
                 \begin{array}{ll}
                   \frac{1}{3}( a_{p^{k-2}}(E_{0,\alpha})+a_{p^{k-2}}(E_{0,\alpha^2})+ a_{p^{k-2}}(E_{0,\alpha^3})    ) & \hbox{if $p\equiv 1\pmod{3}$,} \\
                   (-p)^{k/2-1} & \hbox{if $p\equiv 2 \pmod{3}$,}
                 \end{array}
               \right.
\end{equation}
and $\alpha \in \F_p^*$ is not a cube.

\end{thm}
 Combining Theorems \ref{hyper} and \ref{tracefor3} and using the relation $t_{p^k}(E)=a_{p^k}(E)-p\cdot a_{p^{k-2}}(E)$ then yields the corollary:

\begin{cor}\label{cor1} Let $p\neq 3$ be prime and $k\geq 4$ even. One can alternately express the trace formula as
  $$\tr_k(\Gamma_0(3),p)=\sum_{i=0}^{k/2-2}p^{k-2-i}\sum_{t=2}^{p-1} \Fha{\rho}{\rho^2}{\epsilon}{t}_{p^{k-2-2i}}-p^{k/2-1}(p-2)-\gamma_k(p)-2.$$
\end{cor}

\begin{rem} Because the weight $k$ is even, each $q=p^{k-2-2i}$ automatically satisfies $q\equiv 1 \pmod{3}$, and so Theorem \ref{hyper} can be used in the expression for $\tr_k(\Gamma_0(3),p)$ for all $p\neq 3$.
\end{rem}
\begin{rem}  The function $\gamma_k(p)$ can also be expressed in terms of Gaussian hypergeometric functions as
$$\gamma_k(p)=\left\{
                \begin{array}{ll}
                  -\sum_{i=0 \  3|(k-2-2i) }^{k/2-2}p^{k-2-i} \cdot \Fha{\rho}{\rho^2}{\epsilon}{9\cdot8^{-1}}_{p^{k-2-2i}}+p^{k/2-1} & \hbox{if $p\equiv 1 \pmod{3}$} \\
                  (-p)^{k/2-1} & \hbox{if $p\equiv 2 \pmod{3}$}
                \end{array}
              \right.
$$
and so the trace formula in Corollary \ref{cor1} can be expressed entirely in terms of such functions.
\end{rem}

 One can also use these results to prove  ``inductive trace formulas" as in \cite{fuselier,frech}. Theorem \ref{tracefor3} is particularly well suited for this kind of expression.  A straightforward consequence of Theorem \ref{tracefor3} is the following theorem.

\begin{thm}\label{tracefor3inductivebetter} The trace formula for $p\neq 3$ and $k\geq 6$ even may be written as
$$\tr_k(\Gamma_0(3),p)=p^{k-2}\sum_{t=2}^{p-1}\Fha{\rho}{\rho^2}{\epsilon}{t}_{p^{k-2}}+ p \cdot \tr_{k-2}(\Gamma_0(3),p)+2p-2-\beta_k(p),$$
where
\begin{equation}
\beta_k(p):=\left\{
              \begin{array}{ll}
                0 & \hbox{if $p\equiv 1\pmod{3}$ and $k\equiv 0,1 \pmod{3}$} \\
                -p^{k-2}\cdot \Fha{\rho}{\rho^2}{\epsilon}{9\cdot 8^{-1}}_{p^{k-2}} & \hbox{if $p\equiv 1 \pmod{3}$ and $k\equiv 2 \ \ \pmod{3}$}\\
                2(-p)^{k/2-1} & \hbox{if $p\equiv 2\pmod{3}$.}
              \end{array}
            \right.
\end{equation}
\end{thm}


Many of the same methods used in the $\Gamma_0(3)$ case may be adapted to prove trace formulas for $\Gamma_0(9)$ as well.  We  discuss this in Section \ref{sec:level9} and present a number of formulas for the trace in forms like those above. As an example, we have the following inductive formula.

\begin{thm}\label{tracefor9inductive} Let $k\geq 4$ and $p\equiv 1 \pmod{3}$. Then the trace is given by
\begin{align*}
\tr_k(\Gamma_0(9),p)=p^{k-2}\sum_{t=2 \atop t^3\neq 1}^{p-1}\Fha{\rho}{\rho^2}{\epsilon}{t^3}_{p^{k-2}}+p^{k-2}\Fha{\rho}{\rho^2}{\epsilon}{9 \cdot 8^{-1}}_{p^{k-2}}\\
-4+4p-\delta(k-2)p(p+1)+p\cdot \tr_{k-2}(\Gamma_0(9),p),
\end{align*}
where $\delta(k)=1$ if $k=2$ and 0 otherwise. When $p\equiv 2 \pmod{3}$, we have $\tr_k(\Gamma_0(9),p)=\tr_k(\Gamma_0(3),p)$.
\end{thm}
In fact, when $p\equiv 2 \pmod{3}$, we will see that $\tr_k(\Gamma_0(3^m),p)=\tr_k(\Gamma_0(3),p)$ for every $m$.

Let $q=e^{2\pi i z}$ and recall that the Dedekind eta function is defined to be
$$\eta(z)=q^{\frac{1}{24}}\prod_{n=1}^{\infty}(1-q^n).$$
Then $\eta(3z)^8$ is the unique normalized Hecke eigenform in $S_4(\Gamma_0(9))$ and  we write its Fourier expansion as
$$\eta(3z)^8=\sum b(n) q^n.$$
We will show using trace formula results such as Theorem \ref{tracefor9inductive} that the Fourier coefficients  of $\eta(3z)^8$ are given by the following simple expression.
\begin{cor}\label{fouriercoeff}
Let $p\equiv 1\pmod{3}$, and let $\rho\in \hat{\F}_p^*$ be a character of order three. The $p$th Fourier coefficient  of $\eta(3z)^8$ is given by
$$b(p)=-p^3\left(\bin{\rho^2}{\rho}^3+\bin{\rho}{\rho^2}^3\right)=-p^3\Fha{\rho}{\rho^2}{\epsilon}{9\cdot8^{-1}}_{p^3}.$$
When $p\equiv 2\pmod{3}$, $b(p)=0$.
\end{cor}

In addition, let $V$ be the threefold defined by the following equation:
\begin{equation}\label{eq:threefold}
x^3=y_1y_2y_3(y_1+1)(y_2+1)(y_3+1)
\end{equation}
and let $N(V,p)$ denote the number of $\bar{\F}_p$-points on $V$. Then one can use the results above to show that $V$ is ``modular" in the sense that $N(V,p)$ relates to the Fourier coefficients of $\eta(3z)^8$ by the following expression:
$$b(p)=p^3+3p^2+1-N(V,p).$$

We begin in Section \ref{sec:trace} by stating Hijikata's version of the Eichler-Selberg trace formula for Hecke operators on $S_k(\Gamma_0(\ell))$ where $\ell$ is prime, and then work to simplify this formula into an expression in terms of the number of isomorphism classes of elliptic curves in different isogeny classes. This expression will hold whenever $p\equiv 1 \pmod{\ell}$ or $\leg{p}{\ell}=-1$. We then specialize this formula further to the case where $\ell=3$ in Section \ref{sec:level3} and prove Theorem \ref{tracefor3}.  At the  end of this section we will derive other expressions for the trace on this space, such as Corollary \ref{cor1} and the inductive trace formula in Theorem \ref{tracefor3inductivebetter}. In Section \ref{sec:level9} we  show how methods similar to those in Section \ref{sec:level3} can be used to prove results when $\ell=9$, such as  Theorems \ref{tracefor9} and \ref{tracefor9inductive}. We then  use  these trace formulas to prove Corollary \ref{fouriercoeff}, an explicit expression for the Fourier coefficients of a weight four modular form.  Using this, we show in Section \ref{sec:modularthreefold} that the number of points on the threefold given by equation (\ref{eq:threefold}) can be expressed in terms of the Fourier coefficients of the same modular form.  Finally, in Section \ref{sec:gh}, we prove Theorem \ref{hyper}, an expression for traces of Frobenius of certain elliptic curves in  terms of hypergeometric functions.

While simplifying the expression for $\tr_k(\Gamma_0(3),p)$, we use theorems of Schoof to rewrite sums of class numbers which comes up in the expression in terms of the number of isomorphism  classes of elliptic curves. These theorems only hold when $p, \ell$ satisfy certain congruence properties. Although it is possible when $\ell=3,9$ to reduce the trace formula expression for all values of $p$, this seems to pose a real difficulty for proving trace formulas for general $p, \ell$.

\section{Trace formulas}\label{sec:trace}
\subsection{Hijikata's trace formula}
Let $p,\ell$ be distinct odd primes, and let $k\geq2$ be even. We will specialize the trace formula given by Hijikata in \cite{hij1}  to the case where $T_k(p)$ acts on $S_k(\Gamma_0(\ell))$.  Some preliminary notation is necessary to state the theorem.

For each $s$ in the range $0<s<2\sqrt{p}$, let $t>0,D$ be the unique integers satisfying
\begin{equation}
s^2-4p=t^2D
\end{equation}
 and $D$ is a fundamental discriminant of an imaginary quadratic field. Additionally, for any $d<0$, $d\equiv 0,1 \pmod{4}$, write $h(d):=h(\mathcal{O})$ for the class number of the order $\mathcal{O}\subset \Q(\sqrt{d})$ of discriminant $d$, and  write $\omega(d):=\frac{1}{2}|\mathcal{O}^*|$ for one half of the number of units in that order.  Set $h^*(d):=h(d)/\omega(d)$.

Define  the polynomial $\Phi(X):=X^2-sX+p$ and let $x, y$ be the complex roots of $\Phi(X)$. Define
\begin{equation}
G_{k}(s,p)=:\frac{x^{k-1}-y^{k-1}}{x-y}.
\end{equation}
One can verify that when $k$ is even  $G_{k}(s,p)$ can be alternately expressed as
\begin{equation}\label{g}
G_k(s,p)=\sum_{j=0}^{k/2-1}(-1)^j\bin{k-2-j}{j}p^js^{k-2j-2}.
\end{equation}
Finally, define a function $c(s,f,\ell)$:
\begin{equation}\label{c}
c(s,f,\ell)=\left\{
                        \begin{array}{ll}
                          1+\leg{D}{\ell} & \hbox{if $\ord_\ell(f)=\ord_\ell(t)$,} \\
                          2 & \hbox{if $\ord_\ell(f) < \ord_\ell(t)$.}
                        \end{array}
                      \right.
\end{equation}
Then Hijikata's version of the trace formula yields the following:
\begin{thm}[\cite{hij1}, Theorem 2.2]\label{hijtrace} Let $p, \ell$ be distinct odd primes, and let $k\geq 2$ be even. Then
\begin{equation}\label{traceform}
\tr_k(\Gamma_0(\ell),p)=-\sum_{0<s<2\sqrt{p}}G_k(s,p)\sum_{f|t}\hs{\frac{s^2-4p}{f^2}}c(s,f,\ell)-K(p,\ell)+\delta(k)(1+p)
\end{equation}
where
\begin{eqnarray*}
K(p,\ell)&:=&2+\frac{1}{2}(-p)^{k/2-1}\left(1+\leg{-p}{\ell}\right)H^*(-4p)\\
\delta(k)&:=&\left\{
  \begin{array}{ll}
    1 & \hbox{if $k=2$,} \\
    0 & \hbox{otherwise.}
  \end{array}
\right.
\end{eqnarray*}
\end{thm}
In the following we will often write $H^*(s^2-4p):=\sum_{f|t}\hs{\frac{s^2-4p}{f^2}}$ and $H(s^2-4p):=\sum_{f|t}\h{\frac{s^2-4p}{f^2}}$ for simplicity.


\subsection{Simplifying the formula}
The aim of this section is to rewrite Hijikata's trace formula given in Theorem \ref{hijtrace} in a more convenient form for our purposes by expressing $\tr_k(\Gamma_0(\ell),p)$ in terms of the number of isomorphism classes of elliptic curves with specified torsion. This formula will hold for all $p$ satisfying $p\equiv 1\pmod{\ell}$ or $\leg{p}{\ell}=-1$. In particular, we see that it will hold for all $p\neq 3$ when $\ell=3$.  In the following section we will specialize further to $\ell=3$ to obtain an explicit trace formula.

We begin by eliminating the $c(s,f,\ell)$ term from (\ref{traceform}) Specifically, we show the following:
\begin{lem}\label{removec}
$$\sum_{f|t}\hs{\frac{s^2-4p}{f^2}}c(s,f,\ell)=\left\{
                                                      \begin{array}{ll}
                                                        \left(1+\leg{D}{\ell}\right)H^*(s^2-4p) & \hbox{when $\ell \nmid t$,} \\
                                                        H^*(s^2-4p)+\ell H^*((s^2-4p)/\ell^2)& \hbox{when $\ell | t$.}
                                                      \end{array}
                                                    \right.
$$
\end{lem}

\begin{proof}
Consider first the case where $\ell \nmid  t$. Then $\ord_{\ell}(f)=\ord_{\ell}(t)$ is automatically satisfied, so $c(s,f,\ell)=\left(1+\leg{D}{\ell}\right)$, and the result follows.

When $\ell | t$, we use the following theorem from \cite{cox}:
\begin{thm} [\cite{cox}, Cor. 7.28] \label{coxlem} Let $\mathcal{O}$ be an order of discriminant $d$ in an imaginary quadratic field, and let $\mathcal{O}'\subset \mathcal{O}$ be an order with $[\mathcal{O}: \mathcal{O}']=\iota$. Then
$$h^*(\mathcal{O}')=h^*(\mathcal{O})\cdot \iota \prod_{\ell | \iota, \ \ell \textrm{ prime}}\left(1-\leg{d}{\ell}\frac{1}{\ell}\right).$$
\end{thm}

Substituting the explicit description of $c(s,f,\ell)$ given in  (\ref{c}) and manipulating the terms algebraically gives
{
\begin{eqnarray}\label{ceq}
\sum_{f|t}\hs{\frac{s^2-4p}{f^2}}c(s,f,\ell)&=&\left(1+\leg{D}{\ell}\right)\sum_{f|t, f\nmid t/\ell}\hs{\frac{s^2-4p}{f^2}}+2\sum_{f|t/\ell}\hs{\frac{s^2-4p}{f^2}}\nonumber \\
&=& \sum_{f|t} \hs{\frac{s^2-4p}{f^2}}+\ell\sum_{f|t, f\nmid t/\ell}\hs{\frac{s^2-4p}{f^2}}+\sum_{f|t/\ell}\hs{\frac{s^2-4p}{f^2}} \nonumber\\
&&\ -\ell\left(1-\leg{D}{\ell}\frac{1}{\ell}\right)\sum_{f|t, f\nmid t/\ell}\hs{\frac{s^2-4p}{f^2}}. \label{mid}
\end{eqnarray}
}

By Theorem \ref{coxlem}, the following equality holds
{\small
$$
\ell\left(1-\leg{D}{\ell}\frac{1}{\ell}\right)\sum_{f|t, f\nmid t/\ell}\hs{\frac{s^2-4p}{f^2}}+\ell \sum_{f|t/\ell} \hs{\frac{s^2-4p}{f^2}} =\sum_{f|t} \hs{\ell^2 \frac{s^2-4p}{f^2}}=\sum_{f|t/\ell} \hs{\frac{s^2-4p}{f^2}},$$
}
and so the final term in (\ref{mid}) can therefore be written as
$$ -\ell\left(1-\leg{D}{\ell}\frac{1}{\ell}\right)\sum_{f|t, f\nmid t/\ell}\hs{\frac{s^2-4p}{f^2}}=(\ell-1)\sum_{f|t/\ell}\hs{\frac{s^2-4p}{f^2}},$$
and finally (\ref{mid}) becomes
$$\sum_{f|t}\hs{\frac{s^2-4p}{f^2}}+\ell \sum_{f|t/\ell}\hs{\frac{s^2-4p}{(\ell f)^2}}=H^*(s^2-4p)-\ell H^*((s^2-4p)/(\ell^2))$$

\end{proof}

Using this, the trace formula may be written as
{\small
\begin{eqnarray}
\tr_k(\Gamma_0(\ell),p)&=&-\sum_{0<s<2\sqrt{p}}G_k(s,p)H^*(s^2-4p)\left(1+\leg{s^2-4p}{\ell}\right)\\ \nonumber
&&-\ell\sum_{0<s<2\sqrt{p}, \ell |t}G_k(s,p)H^*\left(\frac{s^2-4p}{\ell^2}\right)- K(p,\ell)+\delta(k)(p+1).
\end{eqnarray}
}
We now rewrite the above equation in terms of the function $H$ instead of $H^*$, so that we may apply Schoof's results counting isomorphism classes of elliptic curves in the next section. Recall that $h^*(d)=h(d)/\omega(d)$, where $\omega(d)=\frac{1}{2}|\mathcal{O}(d)^*|$. Therefore, whenever $d\neq -3,-4$, we have that $h^*(d)=h(d)$. If $s^2-4p=t^2D$ and $D\neq -3,-4$ this implies that $H(s^2-4p)=H^*(s^2-4p)$. If $s^2-4p=-3t^2$, then
\begin{eqnarray*}
H(-3t^2)&=&\sum_{f|t}\h{\frac{-3t^2}{f^2}}=\sum_{f|t, f\neq t}\hs{\frac{-3t^2}{f^2}}+3h^*(-3)\\
&=&\sum_{f|t}\hs{\frac{-3t^2}{f^2}}+2\underbrace{h^*(-3)}_{=1/3}=H^*(-3t^2)+2/3
\end{eqnarray*}
and similarly,  $H(-4t^2)=H^*(-4t^2)+1/2$.

It is left to determine which $s$ satisfy either $s^2-4p=-4t^2$ or $s^2-4p=-3t^2$. By considering the splitting of $p$ in $\Z[i]$ and $\Z\left[\frac{1+\sqrt{-3}}{2}\right]$, we see that the former equality will occur for some $s<2\sqrt{p}$ if and only if $p\equiv 1\pmod{4}$ and the latter if and only if $p\equiv 1 \pmod{3}$. Additionally, by looking at the units in these rings, we see that when $p\equiv 1\pmod{4}$ (respectively $p\equiv 1 \pmod {3}$), there are exactly 2 (resp. 3) values of $s>0$ and $t>0$ for which $s^2-4p=-4t^2$ (resp. $s^2-4p=-3t^2$).

When $p\equiv 1\pmod{4}$ let $a,b$ be positive integers satisfying $p=a^2+b^2$, and similarly when $p\equiv 1 \pmod{3}$, let $c,d$ be positive integers satisfying $p=\frac{c^2+3d^2}{4}$. Then the set of all $(s,t)\in  \mathbb{N}\times \mathbb{N}$ such that $s^2-4p=-4t^2$ is
\begin{equation}
S_4=\{(2a,b),(2b,a)\}
\end{equation}
 and the set of all $(s,t)$ such that $s^2-4p=-3t^2$ is
 \begin{equation}
S_3=\left\{\left( c, d\right),\left(\frac{c+3d}{2},\left| \frac{c-d}{2}\right|\right), \left(\left| \frac{c-3d}{2}\right|,  \frac{c+d}{2}\right)\right\}.
\end{equation}

By a simple congruence argument mod $\ell$, we see that there can be at most one pair $(s,t) \in S_4$ such that $\ell |t$, and similarly for $S_3$. Label the elements of these sets so that in the first case, if $\ell |t$, then $(s,t)=(2a,b)$ and in the second if $\ell |t$ then $(s,t)=(c,d)$.

 Using the solutions in $S_3,S_4$,  we define the following corrective factors
\begin{equation}\label{ep4}
\epsilon_4(p,\ell)=\left\{
                    \begin{array}{ll}
                      \frac{1}{2}(G_k(2a,p)+G_k(2b,p))\left(1+\leg{-4}{\ell}\right) & \hbox{if $p \equiv 1 \pmod{4}$, $\ell \nmid b$,} \\
                      \frac{1}{2}G_k(2b,p)\left(1+\leg{-4}{\ell}\right)+\frac{1}{2}(1+\ell) G_k(2a,p) & \hbox{if $p\equiv 1 \pmod{4}$, $\ell | b$,}\\
                      0 & \hbox{if $p\equiv 3\pmod{4}$}
                    \end{array}
                  \right.
\end{equation}
and
\begin{equation}\label{ep3}
\epsilon_3(p,\ell)=\left\{
                    \begin{array}{ll}
                      \frac{2}{3}\left(G_k(c,p)+G_k(\frac{c+3d}{2},p)+G_k(\frac{c-3d}{2},p)\right)\left(1+\leg{-3}{\ell}\right) & \hbox{if $p \equiv 1\pmod{3}$, $\ell \nmid d$,} \\
                    \frac{2}{3}(G_k(\frac{c+3d}{2},p)+G_k(\frac{c-3d}{2},p))\left(1+\leg{-3}{\ell}\right)
                            +\frac{2}{3}(1+\ell) G_k(c,p) & \hbox{if $p\equiv 1\pmod{3}$, $\ell | d$,}\\
                      0 & \hbox{if $p \equiv 2 \pmod{3}$.}
                    \end{array}
                  \right.
\end{equation}
Using this, the trace formula can be written as
\begin{align}\label{eq1}
\tr_k(\Gamma_0(\ell),p)=&-\sum_{0<s<2\sqrt{p}}G_k(s,p)\left(1+\leg{s^2-4p}{\ell}\right)H(s^2-4p)-\ell \sum_{0<s<2\sqrt{p}, \ell | t}G_k(s,p)H\left(\frac{s^2-4p}{\ell^2}\right)\nonumber\\
&-K(p,\ell)+\ep_4(p,\ell)+\ep_3(p,\ell)+ \delta(k)(p+1).
\end{align}

\subsection{Trace in terms of Elliptic curves}
For an elliptic curve $E$, let $E(\F_p)$ denote the group of $\F_p$-rational points on $E$, and let $E(\F_p)[n]$ denote its $n$-torsion subgroup.
Furthermore, let $\mathcal{I}_p$ denote the set of $\F_p$-isomorphism classes of elliptic curves and write $[E]$ for the isomorphism class containing $E$. Define the sets
 \begin{align*}
I(s)&:=\{\mathcal{C} \in \mathcal{I}_p : \forall E \in \mathcal{C}, |E(\F_p)|=p+1-s\}\\
I_n(s)&:=\{\mathcal{C} \in I(s): \forall E \in \mathcal{C}, \Z/n\Z \subset E(\F_p)[n]\}\\
I_{n\times n}(s)&:=\{\mathcal{C} \in I(s): \forall E \in \mathcal{C}, E(\F_p)[n]\cong \Z/n\Z \times \Z/n\Z\}
\end{align*}
and from these define the quantities $N(s):=|I(s)|$, $N_n(s):=|I_n(s)|$, $N_{n \times n}(s):=|I_{n \times n}(s)|$.

We use the following two theorems of Schoof to rewrite  (\ref{eq1}) in terms  of the above quantities. Although the theorems given in \cite{schoof} hold for curves defined over fields $\F_{p^e}$, we specialize to the prime order case.

\begin{thm}[\cite{schoof}, Thms. 4.6, 4.9] \label{schoof1} Let $s \in \Z$ satisfy $s^2<4p$. Then
$$N(s)=H(s^2-4p).$$
Suppose in addition that $n \in \Z_{\geq 1}$ is odd. Then
 $$N_{n \times n}(s)=\left\{
                       \begin{array}{ll}
                         H\left(\frac{s^2-4p}{n^2}\right) & \hbox{if $p\equiv 1\pmod{n}$ and $s\equiv p+1 \pmod{n^2}$;} \\
                         0 & \hbox{otherwise.}
                       \end{array}
                     \right.
$$
\end{thm}

 If $E$ is an elliptic curve such that $|E(\F_p)|=p+1-s$ then $\Z/n\Z \subset E(\F_p)[n] \iff n \mid \#E(\F_p) \iff  s \equiv p+1 \pmod{n}$. It follows from this  that $N_n(s)=N(s)$ if $s \equiv p+1 \pmod{n}$ and $N_n(s)=0$ otherwise.

We may apply Theorem \ref{schoof1} to replace $H(s^2-4p)$ by $N(s)$ for each $s$ in (\ref{eq1}).  However, since $s$ is not necessarily congruent to $ p+ 1 \pmod{\ell^2}$, we cannot simply replace $H\left(\frac{s^2-4p}{\ell^2}\right)$ by $N_{\ell \times \ell}(s)$ in (\ref{eq1}). Instead, we can use the following lemma when $p\equiv 1\pmod{\ell}$.

\begin{lem}\label{div} Assume that $p\equiv 1 \pmod{\ell}$. Then $\ell^2|s^2-4p \iff \ell^2|p+1-s$ or $\ell^2|p+1+s$.\end{lem}
\begin{proof}
We see that $\ell|p-1 \iff \ell^2|(p-1)^2 \iff p^2-2p+1\equiv 0 \pmod{\ell^2}$. Adding $4p$ to both sides then gives
$$\ell | p-1 \iff (p+1)^2 \equiv 4p \pmod{\ell^2}.$$
Assuming first  that $\ell^2 |s^2-4p$, this implies that  $(p+1)^2\equiv s^2 \pmod{\ell^2}
\implies  (p+1-s)(p+1+s) \equiv 0 \pmod{\ell^2}$. There are now three possibilities. If $\ell^2|p+1-s$ or $\ell^2|p+1+s$ then we are done. Otherwise, it must be that $\ell|p+1-s$ and $\ell|p+1+s$. Then, since we assume throughout that $\ell\neq 2$, this implies that $s\equiv 0 \pmod{ \ell}$. This is a contradiction, since then $0\equiv s^2 \equiv 4p \pmod{ \ell^2}$ and we assumed that $\ell \neq p$.

Conversely, if $\ell^2|p+1-s$ or $\ell^2|p+1+s$, then $(p+1)^2\equiv s^2 \pmod{\ell^2} \implies 4p \equiv s^2 \pmod{\ell^2}$.
\end{proof}

This lemma shows that if $\ell^2$ divides $s^2-4p$ (or equivalently, $\ell | t$) and $p\equiv 1\pmod{\ell}$, then either $s$ or $-s$ satisfies the hypotheses of Theorem \ref{schoof1}.  Therefore, either $H\left(\frac{s^2-4p}{\ell^2}\right)=N_{\ell \times \ell}(s)$ or $H\left(\frac{s^2-4p}{\ell^2}\right)=N_{\ell \times \ell}(-s)$. Since  $s$ and $-s$ cannot both be congruent $ p+1 \pmod{\ell^2}$, it  follows that $H\left(\frac{s^2-4p}{\ell^2}\right)=N_{\ell \times \ell}(s)+N_{\ell \times \ell}(-s)$ and so summing over all $s$ in the range $0<|s|<2\sqrt{p}$ gives

$$\sum_{0<s<2\sqrt{p}, \ell | t}G_k(s,p)H((s^2-4p)/\ell^2)=\sum_{0<|s|<2\sqrt{p}}G_k(s,p)N_{\ell \times \ell}(s).$$

Similarly, if $p\not\equiv 1\pmod{\ell}$ but $\leg{p}{\ell}=-1$, then $\ell \nmid t$ for any $t$ satisfying $s^2-4p=t^2D$. The second sum in (\ref{eq1} ) is empty  and also $N_{\ell\times \ell}(s)=0$ for all $s$ and so we may replace $H\left(\frac{s^2-4p}{\ell^2}\right)$ by $N_{\ell\times\ell}(s)$ in this sum without affecting the value. This shows that when $p\equiv 1\pmod{\ell}$ or $\leg{p}{\ell}=-1$ the trace formula can be written as

\begin{align}\label{eqn2}
\tr_k(\Gamma_0(\ell), p)=&-\sum_{0<|s|<2\sqrt{p}}G_k(s,p)\left(\frac{1}{2}\left(1+\leg{s^2-4p}{\ell}\right)N(s)+\ell N_{\ell\times \ell}(s)\right)\nonumber \\
&-K(p,\ell)+\ep_4(p,\ell)+\ep_3(p,\ell)+\delta(k)(p+1).
\end{align}


\section{Level 3}\label{sec:level3}
We  are now in a position to prove Theorem \ref{tracefor3},  a trace formula for $\ell=3$ and arbitrary prime $p \neq 3$.

\subsection{The case where $p\equiv 1\pmod{3}$}
We first prove the theorem in the case where $p\equiv 1\pmod{3}$. We begin by considering the main term in  (\ref{eqn2}). This term is
$$\sum_{0<|s|<2\sqrt{p}}G_k(s,p)\left(\frac{1}{2}\left(1+\leg{s^2-4p}{3}\right)N(s)+3 N_{3\times 3}(s)\right).$$

For each congruence class of $s \pmod{3}$, consider the term $\frac{1}{2}\left(1+\leg{s^2-4p}{3}\right)N(s)$. When $s \equiv 0 \pmod{3}$, we have $\leg{s^2-4p}{3}=-1$, so $\frac{1}{2}\left(1+\leg{s^2-4p}{3}\right)N(s)=0$, and also $N_3(s)=0$. When $s\equiv 1,2 \pmod{3}$, $\left(1+\leg{s^2-4p}{3}\right)=1$, and the terms in the sum corresponding to $s$ and  $-s$ are $\frac{1}{2}N(s)+\frac{1}{2}N(-s)=N(s)$. Since exactly one of $s,-s$ will be congruent to $ p+1 \pmod{3}$, exactly one of $N_3(s)$ and $N_3(-s)$ will be nonzero and equal to $N(s)$. We may therefore write
$$\frac{1}{2}N(s)+\frac{1}{2}N(-s)=N(s)=N_3(s)+N_3(-s).$$
 The main term is then
$$\sum_{0<|s|<2\sqrt{p}}G_k(s,p)(N_3(s)+3N_{3\times 3}(s)).$$

We next determine the values of $K(p,3), \epsilon_4(p,3)$ and $\epsilon_3(p,3)$. It is clear from the definition of $K(p,3)$ and the fact that $\leg{-p}{3}=\leg{-1}{3}=-1$ that
$$K(p,3)=2.$$
Now, if $p\equiv 1 \pmod{4}$, then $p=a^2+b^2\equiv 1 \pmod{3}$, and 3 must divide exactly one of $a$ or $b$. By our previous convention we assume $3|b$. This gives
\begin{equation}
\epsilon_4(p,3)=\left\{
    \begin{array}{ll}
      2G_k(2a,p) & \hbox{if $p\equiv 1 \pmod{4}$,} \\
      0 & \hbox{if $p\equiv 3 \pmod{4}$.}
    \end{array}
  \right.
\end{equation}
Again, writing $p=\frac{c^2+3d^2}{4}$, a congruence argument shows that $3|d$, and
\begin{equation}
\epsilon_3(p,3)=\frac{2}{3}\left(G_k(c,p)+G_k\left(\frac{c+3d}{2},p\right)+G_k\left(\frac{c-3d}{2},p\right)\right)+2G_k(c,p)
\end{equation}
and the trace formula becomes
\begin{eqnarray*}
\tr_k(\Gamma_0(3),p)&=&-\sum_{0<|s|<2\sqrt{p}}G_k(s,p)(N_3(s)+3N_{3\times 3}(s))-2+2G_k(c,p)+\epsilon_4(p,3)\\
&&+\frac{2}{3}\left(G_k(c,p)+G_k\left(\frac{c+3d}{2},p\right)+G_k\left(\frac{c-3d}{2},p\right)\right)+\delta(k)(p+1).
\end{eqnarray*}

The problem then reduces to parameterizing elliptic curves with 3-torsion and counting isomorphism classes. By changing coordinates so $(0,0)$ is a point of order 3,  any nonsingular elliptic curve $E$ with 3-torsion can be written in the form
\begin{equation}
E: y^2+a_1xy+a_3y=x^3
\end{equation}
with $a_3\neq  0$ (see, for example, Chapter 4 Section 2 in \cite{hus}). The $j$-invariant of such a curve is
\begin{equation}\label{j}
j(E)=\frac{a_1^3(a_1^3-24a_3)^3}{a_3^3(a_1^3-27a_3)}
\end{equation}
 and its discriminant is
\begin{equation}
\Delta(E)=a_3^3(a_1^3-27a_3).
\end{equation}
 By considering the division polynomial $\Psi_3$, it was shown (\cite{miret}, Cor. 5.2)  that when $p\equiv 1\pmod{3}$, $E$ has $E(\F_p)[3]\cong \Z/3\Z \times \Z/3\Z$ if and only if $\Delta(E)$ is a cube in $\F_p$, or equivalently if $a_1^3-27a_3$ is a cube in $\F_p$. We next show how to write any  elliptic curve with $j\neq 0$ in terms of one parameter.

Assume that $j(E)\neq 0$,  then (\ref{j}) implies that $a_1 \neq 0$.  Setting $u=\frac{a_3}{a_1^2}$ and making the change of variables $y \to u^3y$, $x\to u^2x$, gives the isomorphic curve
\begin{equation}
E_t: y^2+txy+t^2y=x^3, \ t=\frac{a_1^3}{a_3}.
\end{equation}
This curve has $j$-invariant $j(E_t)=\frac{t(t-24)^3}{t-27}$ and discriminant $\Delta_t:=\Delta(E_t)=t^6(t^3-27t^2)$. This provides a way of parameterizing all elliptic curves $E$ with $j(E)\neq 0$ and nontrivial 3-torsion.

If $j(E)=0$, then from (\ref{j}) we know that $a_1=0$ or $a_1^3=24a_3$. If in addition $E(\F_p)[3]\cong \Z/3\Z \times \Z/3\Z$ then Lemma 5.6 in \cite{schoof} tells us that there is only one such isomorphism class over $\F_p$. In particular, $E_{24}$ is an elliptic curve over $\F_p$ with $j(E_{24})=0$ and $\Delta_{24}=24^6\cdot 24^2\cdot (-3)=-24^6\cdot 2^6 \cdot 3^3$, a cube. This shows that any such $E$ will be isomorphic to $E_{24}$. In particular, the curve given by $y^2+y=x^3$ is isomorphic to $E_{24}$. Setting $u=24^{-1}a_1$,  and mapping $y\to u^3y $, $x\to u^2x$ shows that when $a_1^3=24a_3$ $E \cong E_{24}$. The curves with $j(E)=0$ and $E(\F_p)[3]\cong \Z/3\Z$  must have $a_1=0$ and are not of the form $E_t$ for any $t$.

Recall that $\mathcal{I}_p$ is the set of isomorphism classes of curves over $\F_p$. Define the following sets
 \begin{align*}
L(s)&:=\{t \in \F_p: \Delta_t\neq 0, |E_t|=p+1-s\}\\
I(s)&:=\{\mathcal{C} \in \mathcal{I}_p: \forall E \in \mathcal{C}, |E|=p+1-s\}\\
I_3(s)&:=\{[E] \in I(s): \Z/3\Z \subset E(\F_p)[3]\}\\
J_3(s)&:=\{[E] \in I(s): E(\F_p)[3]\cong \Z/3\Z, j(E)\neq 0, 1728\}\\
J_{3\times 3}(s)&:=\{[E] \in I(s): E(\F_p)[3]\cong \Z/3\Z \times \Z/3\Z, j(E)\neq 0, 1728\}\\
J^0_3(s)&:=\{[E] \in I(s): E(\F_p)[3]\cong \Z/3\Z, j(E)= 0\}\\
J^0_{3\times 3}(s)&:=\{[E] \in I(s): E(\F_p)[3]\cong \Z/3\Z \times \Z/3\Z, j(E)= 0\}\\
J^{1728}_3(s)&:=\{[E] \in I(s): E(\F_p)[3]\cong \Z/3\Z, j(E)= 1728\}\\
J^{1728}_{3\times 3}(s)&:=\{[E] \in I(s): E(\F_p)[3]\cong \Z/3\Z \times \Z/3\Z, j(E)= 1728\}.\\
\end{align*}
Then $I_3(s)=J_3(s)\cup J_{3\times 3}(s)\cup J_{3}^0(s) \cup J_{3\times 3}^0(s) \cup J_{3}^{1728}(s) \cup J_{3 \times 3}^{1728}(s)$ and by construction this is a union of disjoint sets. Note next that $1728=12^3$ is a cube and that this implies that a curve $E$ with $j$-invariant 1728 has a discriminant that is a cube and therefore has $E(\F_p)[3]\cong \Z/3\Z \times \Z/3\Z$. This shows that $J_{3}^{1728}=\emptyset$.

The goal now is to express the value $|L(s)|$ in terms of the sets above. This is accomplished with the following proposition.
\begin{prop} For every $s$, $L(s)$ satisfies the relationship $$|L(s)|=|J_3(s)|+4|J_{3\times 3}(s)|+|J_{3\times 3}^0(s)|+2|J_{3 \times 3}^{1728}(s)|.$$
\end{prop}
To see this result, define the map
 \begin{equation}
\phi_s:L(s) \to I_3(s) \textrm{ by } t \mapsto [E_t].
\end{equation}
  By the previous discussion, $\phi_s(t) \not\in J_{3}^0(s)$ for any $t$, and so $\phi_s$ maps $L(s)$ onto $J_3(s)\cup J_{3\times 3}(s) \cup J_{3\times 3}^0(s) \cup J_{3 \times 3}^{1728}(s)$, and the following lemma describes the structure of this map.

\begin{lem}  Let $[E]\in I_3(s)$. Then $[E]$ has exactly 1 preimage under  $\phi_s$ when $[E] \in  J_3(s)\cup J_{3\times 3}^0(s)$, exactly 2 preimages when $[E] \in J_{3 \times 3}^{1728}(s)$ and exactly 4 preimages when $[E] \in J_{3\times 3}(s)$.
\end{lem}

\begin{proof}
\case{1} Let $[E] \in  J_{3\times 3}^0(s)$. Then $[E] = [E_{24}]$. Since $0=j(E_t)=\frac{t(t-24)^3}{t-27}$ the only possible preimages of $[E]$ are $t=24,0$. But $\Delta_0 =0$, so $t$ cannot be zero and there is exactly one preimage.
\\
\\
\case{2} Assume now that $p\equiv 1 \pmod{4}$, because otherwise $J_{3 \times 3}^{1728}=\emptyset$, by \cite{schoof}, Lem. 5.6. Let $[E] \in J_{3 \times 3}^{1728}(s)$. Again, $[E] \cong [E_t]$ for some $t$, and $1728=\frac{t(t-24)^3}{t-27}$.  Solving for $t$, we find that the only possible solutions are $t_1=18 + 6\sqrt{3}$, $t_2=18-6\sqrt{3}$. Since $\sqrt{3} \in \F_p$ when $p\equiv 1 \pmod{4}$, both solutions are in $\F_p$. By \cite{schoof} Lem. 5.6, there is only one isomorphism class of curve with $j(E)=1728$ so $\phi_s(t_1)=\phi_s(t_2)=[E]$.
\\
\\
\case{3} We next consider the case where $[E]\in J_3(s) \cup J_{3 \times 3}(s)$, and $j(E)=j_0$. Define the polynomial
$$f(t)=t(t-24)^3-j_0(t-27).$$
This has roots at all $t$ such that $j(E_t)=j_0$. Since $E\cong E_{t_0}$ for some $t_0$, we  know  that there is at least one solution to $f(t)$ in $\F_p$. Recalling that  $\rho$ satisfies $\rho^2+\rho+1=0$ and defining $w$ so that $w^3=(t_0^3-27t_0^2)$, we may factor $f$ over $\bar{\F}_p[x]$ as

{\small
$$  f(t)=(t-t_0)\left(t-\frac{(w-t_0+36)(2w+t_0)}{3w}\right)\left(t-\frac{(\rho w-t_0+36)(2\rho w+t_0)}{3\rho w}\right)\left(t-\frac{(\rho^2 w-t_0+36)(2\rho^2 w+t_0)}{3\rho^2 w}\right).$$}

Since $w\in \F_p$ if and only of $\Delta$ is a cube in $\F_p$ or equivalently $E$ has full 3-torsion, we see that $[E]$ has exactly one preimage when $E(\F_p)[3]\cong \Z/3\Z$. If $w$ is a cube, then  there are four values of $t$ that map to curves isomorphic over $\bar{\F}_p$ to $E$.  These four curves are either isomorphic over $\F_p$ to $E$ or a quadratic twist of $E$. The second case cannot occur because by construction each of the four curves have nontrivial 3-torsion, and so all have their trace of Frobenius congruent to 1 modulo $3$ and a quadratic twist of $E$ would have trace of Frobenius congruent to 2 modulo 3.  Therefore, $[E]$ has four preimages only when $E(\F_p)[3]\cong \Z/3\Z\times \Z/3\Z$.
\end{proof}

The proposition now follows easily from the above lemma. Returning then to the main term of the trace formula, we may write
\begin{eqnarray*}
&&\sum_{0<|s|<2\sqrt{p}}G_k(s,p)(N_3(s)+3N_{3 \times 3}(s))\\
&&\ =\sum_{0<|s|<2\sqrt{p}}G_k(s,p)(\underbrace{|J_3(s)|+4|J_{3\times 3}(s)|+|J_{3\times 3}^0(s)|+2|J_{3\times 3}^{1728}(s)|}_{|L(s)|}+3|J_{3\times 3}^0(s)|+2|J_{3\times 3}^{1728}(s)|+|J_3^0(s)|)\\
&&\ =\sum_{0<|s|<2\sqrt{p}}G_k(s,p)|L(s)|+\sum_{0<|s|<2\sqrt{p}}G_k(s,p)(3|J_{3\times 3}^0(s)|+2|J_{3\times 3}^{1728}(s)|+|J_3^0(s)|)\\
&&\ =\sum_{t\in \F_p, \Delta_t \neq 0} G_k(a(E_t),p)+\sum_{0<|s|<2\sqrt{p}}G_k(s,p)(3|J_{3\times 3}^0(s)|+2|J_{3\times 3}^{1728}(s)|+|J_3^0(s)|).
\end{eqnarray*}
It remains to identify for which $s$ are $J_{3\times 3}^0(s),J_{3\times 3}^{1728}(s),J_3^0(s)$ nonempty. From Schoof \cite{schoof1}, \cite{schoof}, we know that when $p\equiv 1 \pmod{3}$, there are six curves $E$ with $j(E)=0$ and each has $\textrm{End}(E)\cong\Z\left[\frac{1+\sqrt{-3}}{2}\right]$. For each such curve $E$, its trace of Frobenius $s$ therefore satisfies $s^2-4p=-3t^2$ for some $t$. As discussed previously, the six traces $s$ satisfying this equation are $s=\pm c, \pm \frac{c+3d}{2}, \pm \frac{c-3d}{2}$ and for each such $s$, exactly one of $s$ or $-s$ will be congruent to $ p+1 \pmod{3}$, the proper congruence in order to have nontrivial 3-torsion. Of these three, a congruence argument shows that exactly one will further satisfy $s\equiv p+1 \pmod{9}$, and by construction $|s|=c$. Similarly,  If $p\equiv 1\pmod{4}$, the  $E$ such that $j(E)=1728$ have $\textrm{End}(E)\cong \Z[i]$ and as before the trace of Frobenius of such an $E$ will satisfy $s^2-4p=-4t^2$. We use the following lemma from \cite{schoof}
\begin{lem}[\cite{schoof}, Lem. 5.6]\label{numb} Let $\F_p$ be a finite field,
\begin{enumerate}
\item There is at most one elliptic curve $E$ with $j=0$ and $\#E(\F_p)[3]=9$. There is exactly one if and only if $p\equiv 1\pmod{3}$ and this curve has the trace of its Frobenius endomorphism equal to $c$ as above.
\item There is at most one  elliptic curve $E$ with $j=1728$ and $\#E(\F_p)[3]=9$. There is exactly one if and only if $p\equiv 1 \pmod{12}$ and this curve has the trace of its Frobenius endomorphism equal to $2a$.
\end{enumerate}
\end{lem}

Then if $p\equiv 3 \pmod{4}$, $J_{3\times 3}^{1728}(s)=\emptyset$ for all $s$, and if $p\equiv 1\pmod{4}$, $|J_{3\times 3}^{1728}(2a)|=1$ and $J_{3 \times 3}^{1728}(s)=\emptyset$ for all other $s$. Recalling that $\epsilon_4(p,3)=2G_k(2a, p)$, this gives:

\begin{eqnarray*}
\sum_{0<|s|<2\sqrt{p}}G_k(s,p)(N_3(s)+3N_{3 \times 3}(s))&=&\sum_{t\in \F_p, \Delta_t \neq 0} G_k(a(E_t),p) +3G_k(c,p)\\
&&+G_k\left(\frac{c+3d}{2},p\right)+G_k\left(\frac{c-3d}{2}\right)+\epsilon_4(p,3).
\end{eqnarray*}

Finally, we relate this back to the trace.
{\small
\begin{eqnarray*}
\tr_k(\Gamma_0(3),p)&=&-\sum_{0<|s|<2\sqrt{p}}G_k(s,p)(N_3(s)+3N_{3 \times 3}(s))-2+\epsilon_4(p,3)\\
&&+\frac{2}{3}\left(G_k(c,p)+G_k\left(\frac{c+3d}{2},p\right)+G_k\left(\frac{c-3d}{2},p\right)\right)+2G_k(c,p)+\delta(k)(p+1)\\
&=&-\sum_{t\in \F_p, \Delta_t \neq 0} G_k(a_p(E_t),p)-3G_k(c,p)-G_k\left(\frac{c+3d}{2},p\right)-G_k\left(\frac{c-3d}{2}\right)-\epsilon_4(p,3)\\
&&-2+\epsilon_4(p,3)+\frac{2}{3}\left(G_k(c,p)+G_k\left(\frac{c+3d}{2},p\right)+G_k\left(\frac{c-3d}{2},p\right)\right)+2G_k(c,p)+\delta(k)(p+1)\\
&=&-2-\sum_{t\in \F_p, \Delta_t \neq 0} G_k(a_p(E_t),p)-\frac{1}{3}\left(G_k(c,p)+G_k\left(\frac{c+3d}{2},p\right)+G_k\left(\frac{c-3d}{2},p\right)\right)+\delta(k)(p+1).
\end{eqnarray*}
}

This can be simplified with the following lemma.
\begin{lem}\label{ga} Let $E$ be an elliptic curve over $\Q$ and $p$ a prime for which $E$ has good reduction. Recall that
$$L(E,s)=\sum_{n}a_n(E) n^{-s}$$
 is the Hasse-Weil $L$-function of $E$. Then the $p$ power coefficients of $L(E,s)$ can be written explicitly as a function of $a_p(E)$ by
$$a_{p^{k-2}}(E)=G_k(a_p(E),p) \textrm{ when $k\geq 2$.}$$
\end{lem}
\begin{proof}
Recall that we can define $G_k(s,p)$  by
$$G_k(s,p):=\frac{x^{k-1}-y^{k-1}}{x-y}$$
where $x+y=s$ and $xy=p$.

We will show that the function $G_k(a_p(E),p)$ satisfies the same recurrence as the $p$ power coefficients of $L(E,s)$.
 This recurrence for the coefficients $a_{p^e}(f)$ of the $L$-function where $E$ has good reduction at $p$, is
$$a_{p^e}(E)=a_p(E)a_{p^{e-1}}(E)-p\cdot a_{p^{e-2}}(E)$$
and $a_{p^0}(E):=1$.

Explicitly evaluating $G_2(a_p(E),p)$ shows
$$G_2(a_p(E),p)=1=a_{p^0}(E).$$

Now assume that the relation holds for all weights less than $k$. Then in particular
 \begin{eqnarray}
a_{p^{k-3}}(E)&=&G_{k-1}(a_p(E),p) \label{i1}\\
a_{p^{k-4}}(E)&=&G_{k-2}(a_p(E),p). \label{i2}
\end{eqnarray}
 Computing $a_{p^{k-2}}(E)$ using the known recurrence relation and equations (\ref{i1}) and (\ref{i2}) we have
 \begin{eqnarray*}
a_{p^{k-2}}(E)&=&a_p(E)a_{p^{k-3}}(E)-p\cdot a_{p^{k-4}}(E)\\
&=&a_p(E)G_{k-1}(a_p(E),p)-pG_{k-2}(a_p(E),p)\\
&=&a_p(E)\left(\frac{x^{k-2}-y^{k-2}}{x-y}\right)-p\left(\frac{x^{k-3}-y^{k-3}}{x-y}\right).
\end{eqnarray*}
 Since $a_p(E)=x+y$ and $xy=p$, we may replace these in the equation above
\begin{eqnarray*}
 &=&(x+y)\left(\frac{x^{k-2}-y^{k-2}}{x-y}\right)-(xy)\left(\frac{x^{k-3}-y^{k-3}}{x-y}\right)\\
 &=&\frac{x^{k-1}-y^{k-1}}{x-y}\\
 &=&G_k(a_p(E),p)
\end{eqnarray*}
which proves the lemma.

\end{proof}

Let $\alpha \in \F_p^*$ be a noncube. Then one can check that
$$\{|a_p(E_{0,\alpha})|,| a_p(E_{0,\alpha^2})|, |a_p(E_{0, \alpha^3})|\}=\left\{c,  \frac{c+3d}{2}, \left|\frac{c-3d}{2}\right|\right\}$$
so that the final form of the trace formula is
\begin{equation}
\tr_k(\Gamma_0(3),p)=-2-\sum_{t\in \F_p, \Delta_t \neq 0} a_{p^{k-2}}(E_t)-\frac{1}{3}\left(a_{p^{k-2}}(E_{0,\alpha})+a_{p^{k-2}}(E_{0,\alpha^2}) +a_{p^{k-2}}(E_{0,\alpha^3})\right)+\delta(k)(p+1).
\end{equation}


\subsection{The case where  $p\equiv 2\pmod{3}$}\label{sec:level32}
Next, we prove the version of the trace formula for $p\equiv 2\pmod{3}$. The argument follows similarly to the case where $p\equiv 1\pmod{3}$. Also, we assume that $p>3$ since the $p=2$ case is straightforward.  We begin by noting that $\leg{p}{3}=-1$, so that the trace formula in this case is given by (\ref{eqn2}). Also, $\epsilon_3(p,3)=0$ and $\epsilon_4(p,3)=0$ so that the trace formula can be written as

$$\tr_k(\Gamma_0(3),p)=-\frac{1}{2}\cdot \sum_{0<|s|<2\sqrt{p}}G_k(s,p)\left(1+\leg{D}{3}\right)N(s)-K(p,3)+\delta(k)(1+p).$$

Recall that $N_3(s)$ is the number  of isomorphism classes of elliptic curves with trace of Frobenius $s$ and a point of order 3. Since
$$| E| = p +1-s$$
we see that in our case, $N_3(s)=N(s)$ when $s\equiv 0 \pmod{3}$ and $N_3(s)=0$ otherwise.
Still using the relation $s^2-4p=t^2D$, we also find that
\begin{align*}
s\equiv 0 \pmod{3} &\iff D\equiv 1\pmod{3} \iff \left(1+\leg{D}{3}\right)=2\\
s\equiv 1,2 \pmod{3} &\iff D\equiv 2 \pmod{3} \iff \left(1+\leg{D}{3}\right)=0
\end{align*}
so we can write the trace formula as
\begin{eqnarray}
\tr_k(\Gamma_0(3),p)&=&-\sum_{0<|s|<2\sqrt{p}}G_k(s,p)N_3(s)-2-(-p)^{k/2-1}H(-4p)+\delta(k)(1+p)\nonumber \\
&=&-\sum_{0<|s|<2\sqrt{p}}G_k(s,p)N_3(s)-2-G_k(0,p)N(0)+\delta(k)(1+p)\nonumber \\
&=&-\sum_{0\leq |s|<2\sqrt{p}}G_k(s,p)N_3(s)-2+\delta(k)(1+p).
\end{eqnarray}

Again define $E_t : y^2+txy+t^2y=x^3$, which has $(0,0)$ as a point of order 3 and the sets
\begin{align*}
L(s)&:=\{t\in \F_p: \Delta_t \neq 0, |E_t|=p+1-s\}\\
I(s)&:=\{\mathcal{C} \in \mathcal{I}_p : \forall E \in C, |E|=p+1-s\}\\
I_3(s)&:=\{[E]\in I(s): \Z/3\Z \subset E(\F_p)[3]\}
\end{align*}
and consider the map
$$\phi_s:L(s)\to I_3(s) \textrm{ to } t\mapsto [E_t].$$
We will prove the following lemma.
\begin{lem}
Assuming that $s\equiv 0 \pmod{3}$, the map $\phi_s:L(s)\to I_s(s)$ is injective, and when $s\neq 0$ it is a bijection.
\end{lem}
\begin{proof}

When $s\neq 0$, surjectivity is clear, since any elliptic curve with 3-torsion and nonzero $j$-invariant can be written in the form given above, and curves with $j$ invariant equal to 0 will be supersingular. When $s=0$, the isomorphism classes of  curves $E_t^0$  with $j(E_t^0)=0$ given by $E_t^0: y^2+t y=x^3$ are not in the image  of $\phi_0$. Any two curves $E_{t_0}^0$ and $E_{t_1}^0$ in this form will have the Weierstrass forms $E_{t_0}^0: y^2=x^3-(108t_0)^2$ and $E_{t_1}^0: y^2=x^3-(108t_1)^2$. These curves isomorphic over $\F_p$, since all elements  of $\F_p$ are cubes. This shows that when $s=0$, there is exactly one isomorphism class over $\F_p$ that is not in the image of $\phi_0$.

For injectivity, consider first the case where the $j$-invariant is nonzero. Let $[E_{t_0}]\in I_3(s)$ be an  isomorphism class of curve over $\F_p$ with $j$-invariant  $j_0$, and consider its preimage in $L(s)$. As in the case for $p\equiv 1\pmod{3}$, define the polynomial
$$f(t)=t(t-24)^3-j_0(t-27).$$
Then any $t\in L(s)$ mapping to $[E_{t_0}]$ will be a root of $f(t)$. Now, define $w \in \F_p$ to be the unique element of $\F_p$ satisfying
$$w^3=(t_0^3-27t_0^2).$$
Then one can verify that another  root of $f(t)$ is
$$t_1=\frac{(w-t_0+36)(2w+t_0)}{3w}.$$
Therefore, $f(t)$ factors into  two linear terms $(t-t_0)(t-t_1)$ and a quadratic term with discriminant
$$-\frac{3 ((t-36) w+2 (t-27) t)^2}{4 (t-27)^{\frac{2}{3}}t^{\frac{2}{3}}}.$$

Since $p\equiv 2\pmod{3}$, it follows that $-3$ is not a square  in $\F_p$, so the quadratic term does not further simplify over $\F_p$. We have now shown that there are exactly two roots of $f(t)$ in $\F_p$, and so the corresponding curves $E_{t_0}$ and $E_{t_1}$ are isomorphic over $\bar{\F}_p$. We can easily see that they are not isomorphic over $\F_p$ since $E_{t_0}$ will have a quadratic twist defined over $\F_p$ with the same $j$-invariant. Since the trace of Frobenius of this twist is $-s\equiv 0 \pmod{3}$, it will also have 3-torsion, and so can be written as $E_t$ for some $t$. Such a curve will be isomorphic to $E_{t_0}$ over $\bar{\F}_p$ but not $\F_p$, so it must be isomorphic to $E_{t_1}$ over $\F_p$. This shows that $[E_{t_0}]$ has exactly one preimage in $L(s)$.

Now, if $j(E_t)=0$, and  $[E_t]$ is in the  image of $\phi_0$, by a previous discussion in fact $t=24$ and so the map is also injective.

\end{proof}
This lemma shows that $|L(s)|=|I_3(s)|=N_3(s)$ when $s\neq 0$ and $|L(0)|=|I_3(0)|-1 =N(0)-1$. Using this in the trace formula gives
\begin{eqnarray}
\tr_k(\Gamma_0(3),p)&=&-\sum_{0\leq |s|<2\sqrt{p}}G_k(s,p)|L(s)|-G_k(0,p)-2+\delta(k)(1+p)\nonumber \\
&=&-\sum_{t\in \F_p, \Delta_t\neq 0} G_k(a_p(E_t),p)-(-p)^{k/2-1}-2 +\delta(k)(1+p) \nonumber\\
&=&-\sum_{t\in \F_p, \Delta_t\neq 0} a_{p^{k-2}}(E_t)-(-p)^{k/2-1}-2 +\delta(k)(1+p),
\end{eqnarray}
which proves Theorem \ref{tracefor3} in all cases.


\subsection{Proof of Corollary \ref{cor1}}
Corollary \ref{cor1} now follows quickly from Theorems \ref{hyper} and \ref{tracefor3}.
\begin{proof}[Proof of Corollary \ref{cor1}]
Begin with the formula
$$\tr_k(\Gamma_0(3),p)=-\sum_{t\in \F_p, \Delta(E_t) \neq 0} a_{p^{k-2}}(E_t)-\gamma_k(p)-2$$
and use the relation
\begin{equation}\label{ta}
a_{p^{k-2}}(E_t)=t_{p^{k-2}}(E_t)+p\cdot a_{p^{k-4}}(E_t)
\end{equation}
to replace each $a_{p^{k-2}}(E_t)$ to give
$$\tr_k(\Gamma_0(3),p)=-\sum_{t\in \F_p, \Delta(E_t) \neq 0} t_{p^{k-2}}(E_t)-p\cdot \sum_{t\in \F_p, \Delta(E_t) \neq 0}a_{p^{k-4}}(E_t)-\gamma_k(p)-2.$$
 One can see equation  (\ref{ta}) by recalling that
$$a_p(E)=t_p(E)=\alpha+\bar{\alpha}$$ where $\alpha \bar{\alpha}=p$, and that for each $k$, $t_{p^k}(E)=\alpha^k+\bar{\alpha}^k$. Then

$$(\alpha-\bar{\alpha})t_{p^k}(E)=(\alpha-\bar{\alpha})(\alpha^k+\bar{\alpha}^k)=\alpha^{k+1}-\bar{\alpha}^{k+1}-\alpha \bar{\alpha} (\alpha^{k-1}-\bar{\alpha}^{k-1})=\alpha^{k+1}-\bar{\alpha}^{k+1}-p (\alpha^{k-1}-\bar{\alpha}^{k-1})$$
and so

$$t_{p^k}(E)=\frac{\alpha^{k+1}-\bar{\alpha}^{k+1}}{\alpha-\bar{\alpha}}-p \frac{\alpha^{k-1}-\bar{\alpha}^{k-1}}{\alpha-\bar{\alpha}}
=G_{k+2}(a_p(E),p)-p\cdot G_k(a_p(E),p)=a_{p^k}(E)-p\cdot a_{p^{k-2}}(E)$$
where the final equality follows from Lemma \ref{ga}.

Apply again  (\ref{ta}) to each $a_{p^{k-4}}(E_t)$ and so on, until reaching $a_p^0(E_t)=1$.

The result is the following formula:
$$\tr_k(\Gamma_0(3),p)=-\sum_{i=0}^{k/2-2}p^i\sum_{t\in F_p, \Delta(E_t)\neq 0} t_{p^{k-2-2i}}(E_t)-p^{k/2-1}(p-2)-\gamma_k(p)-2.$$
Applying Theorem \ref{hyper} to each $t_{p^{k-2-2i}}(E_t)$ then yields the corollary.
\end{proof}

\subsection{Inductive Trace}\label{sec:inductive}
Now that we have proven Theorem \ref{tracefor3}, we can prove Theorem \ref{tracefor3inductivebetter}, a version of the trace formula which expresses $\tr_k(\Gamma_0(3),p)$ in terms of traces on spaces of smaller weight, as well as an additional inductive formula.

\begin{proof}[Proof of Theorem \ref{tracefor3inductivebetter}]
We show the theorem when $p\equiv 1 \pmod{3}$, but the $p\equiv 2 \pmod{3}$ case follows similarly. We use the relation \ref{ta}
 in order to phrase Theorem \ref{tracefor3} in terms of traces of Frobenius, and then Theorem \ref{hyper} to express this in terms of Gaussian hypergeometric functions.

Replacing each $a_{p^{k-2}}(E_t)$ by $t_{p^{k-2}}(E_t)+p\cdot a_{p^{k-4}}(E_t)$ in the sum then gives
{
\begin{eqnarray*}
\tr_k(\Gamma_0(3),p)&=&-\sum_{{t \in \F_p \atop \Delta(E_t)\neq 0}} t_{p^{k-2}}(E_t)-p\sum_{{t \in \F_p \atop \Delta(E_t)\neq 0}}a_{p^{k-4}}(E_t)-\frac{1}{3}(t_{p^{k-2}}(E_{0,\alpha})+t_{p^{k-2}}(E_{0,\alpha^2})+t_{p^{k-2}}(E_{0,\alpha^3})) \nonumber \\
&&-\frac{1}{3}(p\cdot a_{p^{k-4}}(E_{0,\alpha})+p\cdot a_{p^{k-4}}(E_{0,\alpha^2})+p\cdot a_{p^{k-4}}(E_{0,\alpha^3})) -2\\
&=&-\sum_{{t \in \F_p \atop \Delta(E_t)\neq 0}} t_{p^{k-2}}(E_t)-\frac{1}{3}(t_{p^{k-2}}(E_{0,\alpha})+t_{p^{k-2}}(E_{0,\alpha^2})+t_{p^{k-2}}(E_{0,\alpha^3}))\\
&&+p \cdot \tr_{k-2}(\Gamma_0(3),p)+2p-2\\
&=& p^{k-2}\sum_{t=2}^{p-1}\Fha{\rho}{\rho^2}{\epsilon}{t}_{p^{k-2}}+ p\cdot \tr_{k-2}(\Gamma_0(3),p)\\
&&-\frac{1}{3}(t_{p^{k-2}}(E_{0,\alpha})+t_{p^{k-2}}(E_{0,\alpha^2})+t_{p^{k-2}}(E_{0,\alpha^3}))+2p-2 \textrm{ by Theorem \ref{hyper}}
\end{eqnarray*}
}
Finally, we will show in Lemma \ref{beta} that
$$\frac{1}{3}(t_{p^{k-2}}(E_{0,\alpha})+t_{p^{k-2}}(E_{0,\alpha^2})+t_{p^{k-2}}(E_{0,\alpha^3}))=
\left\{
  \begin{array}{ll}
    0 & \hbox{if $k\equiv 0,1 \pmod{3}$} \\
    -p^{k-2}\cdot \Fha{\rho}{\rho^2}{\epsilon}{9\cdot 8^{-1}}_{p^{k-2}} & \hbox{if $k \equiv 2\ \  \pmod{3}$.}
  \end{array}
\right.
$$
which completes the proof when $p\equiv 1 \pmod{3}$.
\end{proof}


\section{Level 9}\label{sec:level9}
\subsection{Proof of Theorems \ref{tracefor9}, \ref{tracefor9inductive}}

Now we sketch a proof of Theorem \ref{tracefor9inductive}, a trace formula for Hecke operators on $S_k(\Gamma_0(9))$.  Notation and methods are similar to the level 3 case, so we only outline the important differences. Its not hard to see (from the definition of $c(s,f,9)$ given below) that when $p\equiv 2 \pmod{3}$, $\tr_k(\Gamma_0(9),p)=\tr_k(\Gamma_0(3),p)$.  Therefore the level 3 formulas hold in this case. Because  of this, we may assume  throughout that $p\equiv 1 \pmod{3}$.
Applying Hijikata's trace formula results in the following expression:
$$\tr_k(\Gamma_0(9),p)=-\frac{1}{2}\sum_{0<|s|<2\sqrt{p}}G_k(s,p)\sum_{f|t}\hs{\frac{s^2-4p}{f^2}}c(s,f,9)-2+\delta(k)(1+p).$$
The following lemma characterizes the function $c(s,f,9)$.

\begin{prop} Let $s^2-4p=t^2D$ where $D$ is a fundamental discriminant of an imaginary quadratic field and let $f|t$. Let
$$\tau:=\ord_3 t,$$
$$\rho:=\ord_3 f.$$
Then the value of $c(s,f,9)$ is given by:\\
If $\tau=\rho$:
$$c(s,f,9)=\left\{
             \begin{array}{ll}
               2, & \hbox{if $D\equiv 1\pmod{3}$;} \\
               0, & \hbox{if $D \equiv 2 \pmod{3}$;} \\
               0, & \hbox{if $D\equiv 0 \pmod{3}$.}
             \end{array}
           \right.
$$
If $\tau=\rho+1$:
$$c(s,f,9)=\left\{
  \begin{array}{ll}
    5, & \hbox{if $D\equiv 1 \pmod{3}$;} \\
    3, & \hbox{if $D\equiv 2 \pmod{3}$;} \\
    4, & \hbox{if $D\equiv 0 \pmod{3}$.}
  \end{array}
\right.$$
If $\tau> \rho +1$:
$$c(s,f,9)=4.$$
\end{prop}
Because $p\equiv 2 \pmod{3} \implies \tau=\rho$, we have $c(s,f,9)=1+\leg{D}{3}$ when $p\equiv 2 \pmod{3}$. This agrees with the $\ell=3$ case, and the same calculations as in Section \ref{sec:level32} show that $\tr_k(\Gamma_0(3),p)=\tr_k(\Gamma_0(9),p)$. In fact, one can show, using the definition of $c(s,f,N)$ from \cite{hij1} that $c(s,f,3^m)=1+\leg{D}{3}$ for each $m$, and so all of these traces are equal.

Now, as in Lemma \ref{removec}, we remove the $c(s,f,9)$ term from the trace formula by applying Theorem \ref{coxlem}.

\begin{lem} Assume that $p\equiv 1 \pmod{3}$. We can write
$$\sum_{f|t} \hs{\frac{s^2-4p}{f^2}}c(s,f,9)=\left\{
                                               \begin{array}{ll}
                                                 12H^*\left(\frac{s^2-4p}{9}\right) & \hbox{if $3|t$;} \\
                                                 0 & \hbox{otherwise.}
                                               \end{array}
                                             \right.
$$
\end{lem}
\begin{proof}
Consider first the case where $3\nmid t$. Then $\ord_3 f=\ord_3 t=0$. Also, $s^2-4p\equiv 0,2 \pmod{3}$, so if $3\nmid t$ this implies that $t^2D\equiv D\equiv 0,2\pmod{3}$. In either case, $c(s,f,9)=0$ for all $f|t$. This shows that the whole term is 0, agreeing with the lemma.

Now we assume that $3|t$. When $D\equiv 1\pmod{3}$, applying \ref{coxlem} in the second equality below yields the following:
$$2\sum_{f|t, f\nmid t/3} \hs{\frac{s^2-4p}{f^2}}=2\sum_{f|t/3, f\nmid t/9} \hs{\frac{s^2-4p}{(3f)^2}}=\sum_{f|t/3, f\nmid t/9} \hs{\frac{s^2-4p}{f^2}}.$$
This shows that
$$\sum_{f|t} \hs{\frac{s^2-4p}{f^2}}c(s,f,9)=\sum_{f|t/3, f\nmid t/9}\hs{\frac{s^2-4p}{f^2}}\left\{
                                                                       \begin{array}{c}
                                                                         6 \textrm{ if $D\equiv 1 \pmod{3}$} \\
                                                                         3 \textrm{ if $D\equiv 2 \pmod{3}$}\\
                                                                         4 \textrm{ if $D\equiv 0 \pmod{3}$}\\
                                                                       \end{array}
                                                                     \right\}+ 4 \sum_{f|t/9}\hs{\frac{s^2-4p}{f^2}}.$$
Now apply \ref{coxlem} to both terms above. The first one becomes:
$$\sum_{f|t/3, f\nmid t/9}\hs{\frac{s^2-4p}{f^2}}\left\{\begin{array}{c} 6  \\ 3 \\   4 \\ \end{array} \right\}
=\sum_{f|t/3, f\nmid t/p} \hs{\frac{s^2-4p}{(3f)^2}} \left\{\begin{array}{c} 6  \\ 3 \\   4 \\ \end{array} \right\}\cdot \left\{\begin{array}{c} 2  \\ 4 \\   3 \\ \end{array} \right\}=12\sum_{f|t/3 f\nmid t/9} \hs{\frac{s^2-4p}{(3f)^2}}.$$
The second becomes
$$4\sum_{f|t/9}\hs{\frac{s^2-4p}{f^2}}=12\sum_{f|t/9}\hs{\frac{s^2-4p}{(3f)^2}}.$$
Replacing these two quantities into the expression above and combining the sums gives
$$\sum_{f|t} \hs{\frac{s^2-4p}{f^2}}c(s,f,9)=12\sum_{f|t/3}\hs{\frac{s^2-4p}{(3f)^2}}=12H^*\left(\frac{s^2-4p}{9}\right).$$
\end{proof}

Now we have (using the definitions of $\epsilon_4$, $\epsilon_3$, $a$, $c$ in equations (\ref{ep3}), (\ref{ep4}) to rewrite the expressions in $H^*$ in terms of $H$)
\begin{eqnarray*}
\tr_k(\Gamma_0(9),p)&=&-6\sum_{0<|s|<2\sqrt{p},\atop 3|t}G_k(s,p)H^*\left(\frac{s^2-4p}{9}\right)-4+\delta(k)(1+p)\\
&=&-6\sum_{0<|s|<2\sqrt{p},\atop 3|t}G_k(s,p)H\left(\frac{s^2-4p}{9}\right)
+12 \cdot \frac{1}{2} G_k(2a,p)\delta_4(p)+12\cdot \frac{2}{3} G_k(c,p)-4 +\delta(k)(1+p)\\
&=& -12\sum_{0<|s|<2\sqrt{p}} G_k(s,p)N_{3\times 3}(s)+6G_k(2a,p)\delta_4(p)+8G_k(c,p)-4 +\delta(k)(1+p)
\end{eqnarray*}
where $\delta_4(p)=1$ if $p\equiv 1\pmod{4}$ and 0 otherwise.  Keeping notation as in Section \ref{sec:level3}, this is equal to:
\begin{eqnarray*}
&=&-3\sum_{0<|s|<2\sqrt{p}}G_k(s,p)\left(4|J_{3\times 3}(s)|+|J_{3\times 3}^0(s)|+2|J_{3\times 3}^{1728}(s)|+3|J_{3\times 3}^0(s)|+2|J_{3\times 3}^{1728}(s)\right)\\
&& +6G_k(2a,p)+8G_k(c,p)-4 +\delta(k)(1+p)\\
&=&-3\sum_{t\in \F_p, \Delta(E_t)\neq 0, \atop t^3-27t^2 \textrm{ a cube }} G_k(a_p(E_t),p)-G_k(c,p)-4 +\delta(k)(1+p)\\
&=&-3\sum_{t=2, \atop 1-t \textrm{ a cube }}^{p-1} G_k(a_p(E_{27/t}),p)-G_k(c,p)-4 +\delta(k)(1+p)\\
&=&-3\sum_{t=2, \atop 1-t \textrm{ a cube }}^{p-1} G_k\left(p\Fha{\rho}{\rho^2}{\epsilon}{t}_p,p\right)-G_k(c,p)-4 +\delta(k)(1+p).
\end{eqnarray*}

Now apply the transformation law in Theorem 4.4 in Greene \cite{greene}, to write this as
\begin{eqnarray*}
&=&-3\sum_{t=2, \atop 1-t \textrm{ a cube }}^{p-1} G_k\left(p\Fha{\rho}{\rho^2}{\epsilon}{1-t}_p,p\right)-G_k(c,p)-4 +\delta(k)(1+p) \nonumber\\
&=&-3\sum_{t=2, \atop t \textrm{ a cube }}^{p-1} G_k\left(p\Fha{\rho}{\rho^2}{\epsilon}{t}_p,p\right)-G_k(c,p)-4 +\delta(k)(1+p).
\end{eqnarray*}
This gives the  following expression for the trace formula.
\begin{thm}\label{tracefor9}Let $p\equiv 1\pmod{3}$ and $k\geq 4$. Then the trace of the $p$th Hecke operator on $S_k(\Gamma_0(9))$ is given by the expression
$$\tr_k(\Gamma_0(9),p)=-\sum_{t=2, \atop t^3\neq 1}^{p-1}  G_k\left(p\Fha{\rho}{\rho^2}{\epsilon}{t^3}_p,p\right)-G_k(c,p)-4 +\delta(k)(1+p).$$
\end{thm}

From here we may derive a number of expressions as in the level 3 case. For example, using Lemma \ref{ga},  it follows that when $k\geq 4$,
$$\tr_k(\Gamma_0(9),p)=-\sum_{t=1 \atop t^3 \neq 27}^{p-1} a_{p^{k-2}}(E_{t^3})-a_{p^{k-2}}(E_{24})-4.$$
Also, using Equation (\ref{ta}) we can write when $k\geq 4$
\begin{eqnarray*}
\tr_k(\Gamma_0(9),p)&=&\sum_{i=0}^{k/2-2}\sum_{t=2 \atop t^3 \neq 1}^{p-1}p^{k-2-i}\Fha{\rho}{\rho^2}{\epsilon}{t^3}_{p^{k-2-2i}}\\
&&+\sum_{i=0}^{k/2-2}p^{k-2-i}\Fha{\rho}{\rho^2}{\epsilon}{9 \cdot 8^{-1}}_{p^{k-2-2i}}-4-p^{k/2-1}(p-1).
\end{eqnarray*}
Finally, arguing as in Section \ref{sec:inductive} we derive an inductive formula for all $k\geq 6$:
$$\tr_k(\Gamma_0(9),p)=p^{k-2}\sum_{t=2 \atop t^3\neq 1}^{p-1}\Fha{\rho}{\rho^2}{\epsilon}{t^3}_{p^{k-2}}+p^{k-2}\Fha{\rho}{\rho^2}{\epsilon}{9 \cdot 8^{-1}}_{p^{k-2}}-4+4p+p\cdot \tr_{k-2}(\Gamma_0(9),p).$$

\subsection{Proof of Corollary \ref{fouriercoeff}}\label{eta}

Let
$$\eta(3z)^8=\sum b(n) q^n, \ q=e^{2\pi i z}$$
be the Fourier expansion of the unique Hecke eigenform in $S_4(\Gamma_0(9))$. We now prove Corollary \ref{fouriercoeff}, which states that the Fourier coefficients of $\eta(3z)^8$ when $p\equiv 1\pmod{3}$ are given by  the expression
$$b(p)=-p^3\left(\bin{\rho^2}{\rho}^3+\bin{\rho}{\rho^2}^3\right)=-p^3\Fha{\rho}{\rho^2}{\epsilon}{9\cdot 8^{-1}}_{p^3}.$$

\begin{proof}
We actually begin with the alternate trace formula expression from Theorem \ref{tracefor9} and derive the corollary from this. Applying Theorem \ref{tracefor9} with $k=4$ and noting that the dimension for $S_4(\Gamma_0(9))$ is one, we can write
\begin{align*}
b(p)=&-\sum_{t=1}^{p-1}G_4\left(p \Fha{\rho}{\rho^2}{\epsilon}{t},p\right)(\rho^2(t)+\rho(t)+1)-G_4(c,p)-4-3G_4\left(p\Fha{\rho}{\rho^2}{\epsilon}{1},p\right)\\
=&-\sum_{t=1}^{p-1}p^2 \Fha{\rho}{\rho^2}{\epsilon}{t}^2\left(\rho(t)+\rho^2(t)+1\right)-c^2+p^2-3p-1.
\end{align*}

Now compute the term $\sum_{t=1}^{p-1}p^2 \Fha{\rho}{\rho^2}{\epsilon}{t}^2\rho(t)$. Use Definition 3.5 and Theorem 3.6 of Greene \cite{greene}, which in our case (switching $A$ and $B$ in the definition) states that
$$\Fha{\rho}{\rho^2}{\epsilon}{t}=\frac{1}{p}\sum_{y\in \F_p}\rho(y)\rho^2(1-y)\rho(1-ty).$$
Then
\begin{align*}
\sum_{t=1}^{p-1}p^2 \Fha{\rho}{\rho^2}{\epsilon}{t}^2\rho(t)=&\sum_{t=1}^{p-1}p\cdot  \Fha{\rho}{\rho^2}{\epsilon}{t}\sum_{y} \rho(y)\rho^2(1-y)\rho(1-ty)\rho(t)\\
=&p^2 \sum_{y=1}^{p-1} \rho^2(1-y)\left(\frac{1}{p}\sum_{t=1}^{p-1}\Fha{\rho}{\rho^2}{\epsilon}{t}\rho(ty)\rho(1-ty)\right)\\
=&p^2 \sum_{y=1}^{p-1} \rho^2(1-y)\left(\frac{1}{p}\sum_{t=1}^{p-1}\Fha{\rho}{\rho^2}{\epsilon}{ty^{-1}}\rho(t)\rho(1-t)\right).
\end{align*}
Now apply Theorem 3.13 of \cite{greene}, which gives an inductive definition of hypergeometric series.
This gives that the above is equal to
\begin{align*}
p^2\sum_{y=1}^{p-1}\rho^2(1-y)\Fhth{\rho}{\rho^2}{\rho}{\epsilon}{\rho^2}{y^{-1}}=& p^2\sum_{y=1}^{p-1}\rho^2(1-y^{-1})\Fhth{\rho}{\rho^2}{\rho}{\epsilon}{\rho^2}{y}\\
=&p^2\sum_{y=1}^{p-1}\rho(y)\rho^2(1-y)\Fhth{\rho}{\rho^2}{\rho}{\epsilon}{\rho^2}{y^{-1}}\\
=&p^3\Fhf{\rho}{\rho^2}{\rho}{\rho}{\epsilon}{\rho^2}{\epsilon}{1}.
\end{align*}
By the same method we can show that
\begin{equation}
\sum_{t=1}^{p-1}p^2\Fha{\rho}{\rho^2}{\epsilon}{t}^2\rho^2(t)=p^3 \Fhf{\rho}{\rho^2}{\rho^2}{\rho^2}{\epsilon}{\rho}{\epsilon}{1},
\end{equation}
and
\begin{equation}\label{norho}
\sum_{t=1}^{p-1}p^2\Fha{\rho}{\rho^2}{\epsilon}{t}^2=p^3 \Fhf{\rho}{\rho^2}{\epsilon}{\epsilon}{\epsilon}{\rho^2}{\rho}{1}.
\end{equation}
We reduce (\ref{norho}) further using identity 2.15 from \cite{greene}
\begin{equation}\label{greeneid}
\bin{A}{B}\bin{C}{A}=\bin{C}{B}\bin{C\bar{B}}{A\bar{B}}-\frac{p-1}{p^2}B(-1)\delta(A)=\frac{p-1}{p^2}AB(-1)\delta(B\bar{C}),
\end{equation}
where $\delta(A)=1$ if $A=\epsilon$ and 0 otherwise.
When $\chi\neq \epsilon, \rho, \rho^2$, applying equation (\ref{greeneid}) and using Jacobi sum identities gives the equality
$$\bin{\rho \chi}{\chi}\bin{\chi}{\rho \chi}\bin{\rho^2 \chi}{\chi}\bin{\chi}{\rho^2 \chi}=\frac{1}{p^2}.$$ There are $p-4$ such terms.
We must consider the exceptional three cases separately. We have shown so far that:
\begin{align*}
p^3 \Fhf{\rho}{\rho^2}{\rho^2}{\rho^2}{\epsilon}{\rho}{\epsilon}{1}=&\frac{p^4}{p-1}\sum_{\chi} \bin{\rho \chi}{\chi}\bin{\chi}{\rho \chi}\bin{\rho^2 \chi}{\chi}\bin{\chi}{\rho^2 \chi}\\
=&\frac{p^4}{p-1} \bin{\rho}{\epsilon}\bin{\epsilon}{\rho}\bin{\rho^2}{\epsilon}\bin{\epsilon}{\rho^2}
+\frac{p^4}{p-1}\bin{\rho^2}{\rho}\bin{\rho}{\rho^2}\bin{\epsilon}{\rho}\bin{\rho}{\epsilon}\\
&+\frac{p^4}{p-1}\bin{\epsilon}{\rho^2}\bin{\rho^2}{\epsilon}\bin{\rho}{\rho^2}\bin{\rho^2}{\rho} +\frac{p^2(p-4)}{p-1}\\
=&\frac{p^2(p-4)}{p-1}+ \frac{1}{p-1}+ \frac{2p^2}{p-1}\bin{\rho^2}{\rho}\bin{\rho}{\rho^2}\\
=&\frac{p^3-4p^2+2p+1}{p-1}=p^2-3p-1.
\end{align*}
Combining these results, we have
$$b(p)=-p^3\Fhf{\rho}{\rho^2}{\rho^2}{\rho^2}{\epsilon}{\epsilon}{\rho}{1}-p^3\Fhf{\rho^2}{\rho}{\rho}{\rho}{\epsilon}{\epsilon}{\rho^2}{1}-c^2.$$
Reducing this further, we use 2.15 again to write
$$p^3\Fhf{\rho}{\rho^2}{\rho^2}{\rho^2}{\epsilon}{\epsilon}{\rho}{1}=p^3\bin{\rho^2}{\rho}\Fhth{\rho^2}{\rho^2}{\rho^2}{\epsilon}{\epsilon}{1}
-p^2\bin{\rho}{\rho^2}{\bin{\rho}{\rho^2}}$$
and similarly
$$p^3\Fhf{\rho^2}{\rho}{\rho}{\rho}{\epsilon}{\epsilon}{\rho^2}{1}=p^3\bin{\rho}{\rho^2}\Fhth{\rho}{\rho}{\rho}{\epsilon}{\epsilon}{1}
-p^2\bin{\rho^2}{\rho}{\bin{\rho^2}{\rho}}.$$

Now, we can evaluate these hypergeometric series using Theorem 4.35 from \cite{greene}. Using this, we have
$$\Fhth{\rho^2}{\rho^2}{\rho^2}{\epsilon}{\epsilon}{1}=\bin{\rho^2}{\rho}\bin{\rho^2}{\rho}-\frac{1}{p}\bin{\rho}{\rho^2}$$
$$\Fhth{\rho}{\rho}{\rho}{\epsilon}{\epsilon}{1}=\bin{\rho}{\rho^2}\bin{\rho}{\rho^2}-\frac{1}{p}\bin{\rho^2}{\rho}.$$

So
\begin{align*}
b(p)=&-p^3\bin{\rho^2}{\rho}^3+p^2\bin{\rho^2}{\rho}\bin{\rho}{\rho^2}+p^2\bin{\rho}{\rho^2}\bin{\rho}{\rho^2}\\
&-p^3\bin{\rho}{\rho^2}^3+p^2\bin{\rho}{\rho^2}\bin{\rho^2}{\rho}+p^2\bin{\rho^2}{\rho}\bin{\rho^2}{\rho}-c^2\\
=&-p^3\left(\bin{\rho}{\rho^2}^3+\bin{\rho^2}{\rho}^3\right)+p^2\left(\bin{\rho}{\rho^2}+\bin{\rho^2}{\rho}\right)^2-c^2.
\end{align*}

Finally, recall that $c$ is  the trace of Frobenius of the curve $E: y^2+y=x^3$, which we computed in equation (\ref{fro})  and is given by
\begin{equation}\label{csum}
c=-\frac{1}{p}G_{\frac{p-1}{3}}^3-\frac{1}{p}G_{\frac{2(p-1)}{3}}^3=-p\bin{\rho}{\rho^2}-p\bin{\rho^2}{\rho}.
\end{equation}
Using this in the equation above, we have that
$$b(p)=-p^3\left(\bin{\rho}{\rho^2}^3+\bin{\rho^2}{\rho}^3\right).$$

For the second equality in Corollary \ref{fouriercoeff}, let $\alpha=-p\bin{\rho^2}{\rho}$. Then by equation (\ref{csum}), $t_p(E)=\alpha+\bar{\alpha}$ and $\alpha\bar{\alpha}=p$.   It follows then that $t_{p^3}(E)=\alpha^3+\bar{\alpha}^3=b(p)$.  Theorem \ref{hyper} now implies that $$b(p)=t_{p^3}(E)=-p^3\Fha{\rho}{\rho^2}{\epsilon}{9\cdot 8^{-1}}_{p^3}.$$

\end{proof}

\section{A Modular Threefold}\label{sec:modularthreefold}
Now let $V$ be the threefold defined by the equation
$$x^3=y_1y_2y_3(y_1+1)(y_2+1)(y_3+1),$$
and let $N(V,p)$ denote the number of $\bar{\F}_p$-points on $V$. Then we will show that $V$ is ``modular" in the sense that the number of points on $V$ can be expressed in terms of the Fourier coefficients of a modular form. In particular, the function $\eta(3z)^8=\sum b(n)q^n$ discussed in Section \ref{eta} has Fourier coefficients given by the expression
\begin{equation}
b(p)=p^3+3p^2+1-N(V,p).
\end{equation}

We consider first the case where $p\equiv 1 \pmod{3}$. Recall that $E: y^2+y=x^3$ is an elliptic curve with $j$-invariant 0 and trace of Frobenius $t_p(E)=c=-p\cdot \Fha{\rho}{\rho^2}{\epsilon}{9\cdot 8^{-1}}_p$, where $c$ satisfies $c^2-4p=-3d^2$ and $3|d$. (The expression $t_p(E)=-p \cdot \Fha{\rho}{\rho^2}{\epsilon}{9 \cdot 8^{-1}}$ follows from the fact that $E$ is isomorphic to the curve $y^2+24xy+24^2y^2=x^3$ and Theorem \ref{hyper}.) Then $\tr_4(\Gamma_0(9),p)=b(p)$ and by Corollary \ref{fouriercoeff},
$$b(p)= -p^3\left(\bin{\rho^2}{\rho}^3+\bin{\rho}{\rho^2}^3\right)=t_{p^3}(E),$$
when $p\equiv 1\pmod{3}$.  We may rewrite this expression as
$$b(p)=t_{p^3}(E)=t_p(E)^3-3p t_p(E)=c^3-4cp=c \cdot \left(\frac{-c+3d}{2}\right)\cdot\left(\frac{-c-3d}{2}\right).$$
Let $\beta \in \F_p^*$ be a noncube and consider the three curves
\begin{eqnarray*}
E_1:& y^2+y&=x^3\\
E_2:& y^2+\beta y&= x^3\\
E_3:& y^2+\beta^2 y&= x^3.
\end{eqnarray*}
Then $E_2$ and $E_3$ are two cubic twists of $E_1$. The theory of elliptic curves tells us that $t_p(E_1)=c$ as before and $t_p(E_2)=\frac{-c+3d}{2}$, $t_p(E_3)=\frac{-c-3d}{2}$. This gives
$$b(p)=t_p(E_1)t_p(E_2)t_p(E_3).$$
Write $N_i=p+1-t_p(E_i)$ for the number  of projective points of $E_i$. We may write $b(p)$ as an expression in these by
\begin{equation}\label{eq:traceell}
b(p)=p^3+1-N_1N_2N_3.
\end{equation}
Define the sets:
\begin{align*}
W=&\{(y_1,y_2,y_3,x)\in \F_p^4: y_1y_2y_3(1+y_1)(1+y_2)(1+y_3)=x^3\}\\
V_1=&\{((y_1,x_1),(y_2,x_2),(y_3,x_3))\in \F_p^6: y_i^2+y_i=x_i^3,\  i=1,2,3 \textrm{ and } y_i \neq 0,-1\}\\
V_2=&\{((y_1,x_1),(y_2,x_2),(y_3,x_3))\in \F_p^6: y_i^2+y_i=\beta x_i^3,\  i=1,2,3 \textrm{ and } y_i \neq 0,-1\}\\
V_3=&\{((y_1,x_1),(y_2,x_2),(y_3,x_3))\in \F_p^6: y_i^2+y_i=\beta^2 x_i^3,\  i=1,2,3 \textrm{ and } y_i \neq 0,-1\}\\
V_4=&\{((y_1,x_1),(y_2,x_2),(y_3,x_3))\in \F_p^6: y_i^2+y_i=\beta^{j_i}x_i^3,\  i=1,2,3 \  j_1+j_2+j_3=3, j_{i}\neq j_{k} \textrm{ and } y_i \neq 0,-1\}.
\end{align*}

Then its not hard to see that $\#V_i=(N_i-3)^3$ for $i=1,2,3$ and $\#V_4=6(N_1-3)(N_2-3)(N_3-3)$ (subtracting 2 from $N_i$ to remove the points corresponding to $y=0,-1$ and an extra one to count only affine points). The number of elements in $W$ is then expressible as
\begin{align*}
\#W=&(\#V_1+\#V_2+\#V_3+\#V_4)/9+6(p-2)^2+12(p-2)+8\\
=&\left((N_1-3)^3+(N_2-3)^3+(N_3-3)^3+6(N_1-3)(N_2-3)(N_3-3)\right)/9+6(p-2)^2+12(p-2)+8\\
=&N_1N_2N_3-1.
\end{align*}
Using this expression in equation (\ref{eq:traceell}), we see that the trace is  given by
\begin{equation}\label{eq:traceell2}
b(p)=p^3-\#W.
\end{equation}

Now we compute the value $N(W,p)$, the number of projective points on $V$, in terms of $\#W$.  Begin by homogenizing the equation:
$$y_1y_2y_3(y_1+z)(y_2+z)(y_3+z)=x^3z^3.$$
The points corresponding to $z=0$ are on the curve
$$y_1^2y_2^2y_3^3=0.$$

First we count the points where $x\neq 0$, (so fix $x=1$). At least one $y_i$ must be zero and the other two can be anything. There are exactly $3(p-1)^2+3(p-1)+1$ possible choices for $y_1,y_2,y_3$.

Now, we count the points corresponding to $x=0$. We choose  one of the $y_i$ to be zero (three choices) and there are $\frac{(p-1)^2}{(p-1)}$ values for the other $y_i$'s. There are another 3 points corresponding to when exactly two of the $y_i$ are 0. So the number  of projective points is
$$N(V,p)=\#W+ 3(p-1)^2+3(p-1)+1+3(p-1)+3=\#W+1+3p^2.$$
Finally, combining the above with equation \ref{eq:traceell2} it follows that, when $p\equiv 1\pmod{3}$,
$$b(p)=p^3-\#W=p^3+3p^2+1-N(V,p).$$

If however, $p\equiv 2\pmod{3}$, then every element of $\F_p$ is a cube, and so for any choice of $y_1,y_2,y_3$  there is a unique $x$ satisfying equation \ref{eq:threefold}.  There are $p^3$ such choices for the $y_i's$, so $\#W=p^3$. Since $b(p)=0$ for all $p\equiv 2\pmod{3}$, it follows that
$$b(p)=0=p^3-\#W=p^3+3p^2+1-N(V,p).$$


\section{Expressing the number of points on $E$ as a Gaussian hypergeometric function}\label{sec:gh}
Again we have $q=p^e \equiv 1 \pmod{3}$, $p>3$ a prime. Let $E_{a_1,a_3}$ be the curve $y^2+a_1xy+a_3y=x^3$, where $a_1,a_3 \in \Z$ and assume that $E_{a_1,a_3}$ has good reduction modulo $p$.  Write $\tilde{E}_{a_1,a_3}$ for the reduction modulo $p$ and $\#\tilde{E}_{a_1,a_3}(\F_q)$ for the number of projective points of $\tilde{E}_{a_1,a_3}$ in $\F_q$. We next prove Theorem \ref{hyper}, which expresses the trace of the Frobenius map on $\tilde{E}_{a_1,a_3}(\F_q)$ as a special value of a Gaussian hypergeometric function. We begin as in \cite{fuselier} by expressing the number of points as an exponential sum. Note that the cited lemmas from \cite{fuselier} are given there for the prime field case, but they can be easily extended to $q$ a power of a prime.

\begin{proof}[Proof of Theorem \ref{hyper}]

If we let $$P(x,y)=y^2+a_1xy+a_3y-x^3$$
 then
$$\#\tilde{E}_{a_1,a_3}(\F_q)-1=\#\{(x,y)\in \F_q \times \F_q: P(x,y)=0\}.$$
  Define the additive character $\theta: \F_q \to \C^*$ by
\begin{equation}
\theta(\alpha)=\zeta^{\tr(\alpha)}
\end{equation}
 where  $\zeta=e^{2\pi i/p}$ and $\tr: \F_q \to \F_p$ is the trace map, ie $\tr(\alpha)=\alpha + \alpha^p+\alpha^{p^2}+...+\alpha^{p^{e-1}}$. We will repeatedly use the elementary identity \cite{ire}
\begin{equation}\label{thetasum}
\sum_{z\in \F_q}\theta(zP(x,y))=\left\{
                                    \begin{array}{ll}
                                      q & \hbox{if $P(x,y)=0$,} \\
                                      0 & \hbox{if $P(x,y)\neq0$.}
                                    \end{array}
                                  \right.
\end{equation}

Using this  we  write
\begin{eqnarray*}
q(\#\tilde{E}_{a_1,a_3}(\F_q)-1)&=&\sum_{z \in \F_q}\sum_{x,y \in \F_q}\theta(z P(x,y))\\
&=& q^2+(q-1)+\underbrace{\sum_{z \in \F_q^*}\sum_{x\in \F_q^*}\theta(-zx^3)}_{B}
+\underbrace{\sum_{z\in \F_q^*}\sum_{y\in \F_q^*}\theta(zy^2)\theta(za_3y)}_C+\underbrace{\sum_{x,y,z \in \F_q^*}\theta(z P(x,y))}_D.
\end{eqnarray*}

We can compute these sums using the following lemma from \cite{fuselier}:
\begin{lem}[\cite{fuselier}, Lemma 3.3]\label{theta} For all $\alpha \in \F_q^*$,
$$\theta(\alpha)=\frac{1}{q-1}\sum_{m=0}^{q-2}G_{-m}T^m(\alpha),$$ where $T$ is a fixed generator of the character group and $G_{-m}$ is the Gauss sum $G_{-m}:=G(T^{-m})=\sum_{x\in \F_q}T^{-m}(x)\theta(x)$.
\end{lem}

Computing $B$: Use Lemma \ref{theta} to replace $\theta(-zx^3)$, and then apply the orthogonality relation (\ref{thetasum})
\begin{eqnarray*}
B&=&\sum_{z,x\in \F_q^*}\frac{1}{q-1}\sum_m G_{-m}T^m(-x^3)T^m(z)=\frac{1}{q-1}\sum_m G_{-m}\sum_{x\in \F_q^*}T^m(-x^3)\sum_{z\in \F_q^*}T^m(z)\\
&=&\sum_{x\in \F_q^*}G_0=-(q-1).
\end{eqnarray*}

Computing $C$:
\begin{eqnarray*}
C&=&\sum_{z,y \in \F_q^*}\frac{1}{(q-1)^2}\sum_{k,m}G_{-k}G_{-m}T^{k+m}(z)T^{2k+m}(y)T^m(a_3)\\
&=&\frac{1}{(q-1)^2}\sum_{k,m}G_{-k}G_{-m}T^m(a_3)\sum_{z\in \F_q^*}T^{k+m}(z)\sum_{y\in \F_q^*}T^{2k+m}(y)
\end{eqnarray*}

We see here that $k+m=0\implies m=-k$ and $2k+m=0 \implies k=m=0$. So this becomes $G_0^2=1$.

Computing $D$:
\begin{eqnarray*}
D&=&\sum_{x,y,z\in \F_q^*}\frac{1}{(q-1)^4}\sum_{j,k,l,m}G_{-j}G_{-k}G_{-l}G_{-m}T^{j+k+l+m}(z)T^{2j+k+l}(y)T^{k+3m}(x)T^k(a_1)T^{l}(a_3)T^m(-1)\\
&=&\frac{1}{(q-1)^4}\sum_{j,k,l,m}G_{-j}G_{-k}G_{-l}G_{-m}T^m(-1)T^k(a_1)T^l(a_3)\sum_{x \in \F_q^*}T^{k+3m}(x)\sum_{y \in \F_q^*}T^{2j+k+l}(y)\sum_{z\in \F_q^*}T^{j+k+l+m}(z)
\end{eqnarray*}
Solving each of these equations $j+k+l+m=0$, $k+3m=0$, $2j+k+l=0$ gives $k=-3m$ and $l=m=j$. Plugging this  in we have
$$D=\frac{1}{q-1}\sum_m G_{-m}^3G_{3m}T^{-3m}(a_1)T^m(-a_3).$$
Putting this all together then gives
$$q(\#\tilde{E}_{a_1,a_3}(\F_q)-1)=q^2+1+\frac{1}{q-1}\sum_m G_{-m}^3G_{3m}T^{-3m}(a_1)T^m(-a_3)$$
and so $\#\tilde{E}_{a_1,a_3}(\F_q)=1+q+\frac{1}{q}+\frac{1}{q(q-1)}\sum_mG_{-m}^3G_{3m}T^{-3m}(a_1)T^m(-a_3)$ and finally

$$t_q(E_{a_1,a_3})=q+1-\#\tilde{E}_{a_1,a_3}(\F_q)=-\frac{1}{q}-\frac{1}{q(q-1)}\sum_m G_{-m}^3G_{3m}T^{-3m}(a_1)T^m(-a_3).$$

In order to write this as a finite field hypergeometric function, we use the fact that if $T^{m-n}\neq \epsilon$, then
\begin{equation}\label{jacgauss}
 \left({ T^m \atop T^n } \right)=\frac{G_mG_{-n}T^n(-1)}{G_{m-n}q}.
\end{equation}
This is  a restating of the classical identity $G(\chi_1)G(\chi_2)=J(\chi_1,\chi_2)G(\chi_1\chi_2)$ which holds whenever $\chi_1\chi_2$ is a primitive character. We also use the  Davenport-Hasse relation for $q\equiv 1\pmod{3}$. We state the general result as well as the case needed here.

\begin{thm}[Davenport-Hasse Relation \cite{lang}]\label{hasdavgen} Let  $m$ be a positive integer and let $q=p^e$ be a prime power such that $q\equiv 1 \pmod{m}$. Let $\theta$ be the additive character on $\F_q$ defined by $\theta(\alpha)=\zeta^{\tr \alpha}$, where $\zeta=e^{2 \pi i /p}$. For multiplicative characters $\chi, \psi \in \widehat{\F}_q^*$ we have
$$\prod_{\chi^m=1}G(\chi \psi)=-G(\psi^m)\psi(m^{-m})\prod_{\chi^m=1}G(\chi).$$
\end{thm}

\begin{cor}[Davenport-Hasse for $q\equiv 1\pmod{3}$]\label{hasdav} If $k\in \Z$ and $q$ satisfies $q\equiv 1 \pmod{3}$ then $G_kG_{k+\frac{q-1}{3}}G_{k+\frac{2(q-1)}{3}}=qT^{-k}(27)G_{3k}$. \end{cor}

First we use  Corollary \ref{hasdav} to write $G_{3m}=G_mG_{m+\frac{q-1}{3}}G_{m+\frac{2(q-1)}{3}}\frac{T^m(27)}{q}$,  giving

$$t_q(E_{a_1,a_3})=-\frac{1}{q}-\frac{1}{q^2(q-1)}\sum_mG_{-m}^3G_mG_{m+\frac{q-1}{3}}G_{m+\frac{2(q-1)}{3}}T^m(27)T^{-3m}(a_1)T^m(-a_3).$$

Next, make the substitution $G_mG_{-m}=qT^m(-1)$, which holds whenever $m\neq 0$. For $m=0$, we write $G_mG_{-m}=T^m(-1)=qT^m(-1)-(q-1)T^m(-1)$:
$$t_q(E_{a_1,a_3})=-\frac{1}{q}-\frac{1}{q(q-1)}\sum_m G_{-m}^2 G_{m+\frac{q-1}{3}}G_{m+\frac{2(q-1)}{3}}T^{-3m}(a_1)T^m(27a_3)+\frac{G_{\frac{q-1}{3}}G_{\frac{2(q-1)}{3}}}{q^2}.$$

For the last term above, note that $G_{\frac{q-1}{3}}G_{\frac{2(q-1)}{3}}=q$ and cancel with the first term, giving
$$t_q(E_{a_1,a_3})=-\frac{1}{q(q-1)}\sum_m G_{-m}^2 G_{m+\frac{q-1}{3}}G_{m+\frac{2(q-1)}{3}}T^{-3m}(a_1)T^m(27a_3).$$

Now apply equation (\ref{jacgauss}) to write {\small $G_{m+\frac{q-1}{3}}G_{-m}=\left( T^{m+\frac{q-1}{3}} \atop T^m \right)G_{\frac{q-1}{3}}q T^m(-1)$ } and \\ {\small $G_{m+\frac{2(q-1)}{3}}G_{-m}=\left(T^{m+\frac{2(q-1)}{3}} \atop T^m\right)G_{\frac{2(q-1)}{3}}qT^m(-1)$.} Plugging this in yields

$$t_q(E_{a_1,a_3})=-\frac{q\left(G_{\frac{q-1}{3}}G_{\frac{2(q-1)}{3}}\right)}{q-1}\sum_m \left( T^{m+\frac{q-1}{3}} \atop T^m \right)\left(T^{m+\frac{2(q-1)}{3}} \atop T^m\right) T^{-3m}(a_1)T^m(27a_3).$$

Again use the fact that $G_{\frac{q-1}{3}}G_{\frac{2(q-1)}{3}}=q$ to get
\begin{eqnarray*}
t_q(E_{a_1,a_3})&=&-\frac{q^2}{q-1}\sum_m \left( T^{m+\frac{q-1}{3}} \atop T^m \right)\left(T^{m+\frac{2(q-1)}{3}} \atop T^m\right) T^m(27a_1^{-3}a_3)\\
&=&-q \cdot {}_2F_1\left( \left. {\begin{array}{cc} T^{\frac{q-1}{3}} & T^{\frac{2(q-1)}{3}}\\ \ & \epsilon \end{array}} \right| 27 a_1^{-3}a_3\right)_q.
\end{eqnarray*}
\end{proof}
When $\tilde{E}_{a_1,a_3}=\tilde{E}_t$ this expression reduces to
$$t_q(E_t)=-q\cdot \Fha{\rho}{\rho^2}{\epsilon}{\frac{27}{t}}_q$$
where $\rho$ is a character of order 3.

Finally, we prove the following lemma, which will allow us to represent the sums of the trace  of Frobenius of elliptic curves with $j$-invariant 0 in terms of Gaussian hypergeometric functions.
\begin{lem}\label{beta} When $p\equiv 1 \pmod{3}$ and $\alpha$ is not a cube in $\F_p^*$,
$$\frac{1}{3}(t_{p^{k-2}}(E_{0,\alpha})+t_{p^{k-2}}(E_{0,\alpha^2})+t_{p^{k-2}}(E_{0,\alpha^3}))=
\left\{
  \begin{array}{ll}
    0 & \hbox{if $k\equiv 0,1 \pmod{3}$} \\
    -p^{k-2}\cdot \Fha{\rho}{\rho^2}{\epsilon}{9\cdot 8^{-1}}_{p^{k-2}} & \hbox{if $k \equiv 2\ \ \pmod{3}$.}
  \end{array}
\right.$$
\end{lem}
\begin{proof}
As in the proof of Theorem \ref{hyper}, set $P(x,y)=y^2+\alpha^i y-x^3$ and compute
{\small
\begin{eqnarray*}
q(\#\tilde{E}_{0,\alpha^i}(\F_q)-1)&=&\sum_{z \in \F_q}\sum_{x,y \in \F_q} \theta(z P(x,y))\\
&=& q^2+ (q-1)-(q-1)+1+\sum_{x,y,z \in  \F_q^*} \theta( z P(x,y))\\
&=&q^2+1+\frac{1}{(q-1)^3}\sum_{j,k,l}G_{-j}G_{-k}G_{-l}T^k(\alpha^i)T^l(-1)\sum_{z\in \F_q^*}T^{j+k+l}(z)\sum_{x \in  \F_q^*}T^{3l}(x) \sum_{y \in \F_q^*}T^{2j+k}(y)
\end{eqnarray*}
}
The terms above will be nonzero when $3l=0$ and $j=k=l$. Plugging this in above gives
\begin{eqnarray*}
q(\#\tilde{E}_{0,\alpha^i}(\F_q)-1)&=&q^2+1+\sum_{j=0,\frac{q-1}{3}, \frac{2(q-1)}{3}}G_{-j}^3 T^j(\alpha^i)
\end{eqnarray*}
And so
\begin{equation}\label{fro}
t_q(E_{0,\alpha^i})=-\frac{1}{q}-\frac{1}{q}\sum_{j=0,\frac{q-1}{3}, \frac{2(q-1)}{3}} G_{-j}^3 T^j(\alpha^i)
\end{equation}
and summing over all three traces then gives
$$t_q(E_{0,\alpha})+t_q(E_{0,\alpha^2})+t_q(E_{0,\alpha^3})=-\frac{3}{q}-\frac{1}{q}\sum_{j=0,\frac{q-1}{3},\frac{2(q-1)}{3}}G_{-j}^3\left(T^j(\alpha)+T^{2j}(\alpha)+T^{3j}(\alpha)\right).$$

Now let $q=p^{k-2}$ and let $g\in \F_{p^{k-2}}^*$ generate  the group. Since $\alpha \in \F_p^*$, we know that $\alpha^{p-1}=1$, and so $\alpha=g^{a \frac{p^{k-2}-1}{p-1}}=g^{a(p^{k-3}+p^{k-4}+...+1)}$ for some integer $a$. Since $p\equiv 1 \pmod{3}$, it follows that
$p^{k-3}+p^{k-4}+...+1\equiv k-2 \pmod{3}$.

By the above argument, when $k\equiv 0,1 \pmod{3}$, $\alpha$ is not a cube in $\F_{p^{k-2}}^*$ (recall that $\alpha$ was initially chosen as a noncube in $\F_p^*$). Therefore
$$T^j(\alpha)+T^{2j}(\alpha)+T^{3j}(\alpha)=0$$
when $j=\frac{q-1}{3}$ and $ \frac{2(q-1)}{3}$. The only nonzero term then is when $j=0$, and computing this gives
$$t_q(E_{0,\alpha})+t_q(E_{0,\alpha^2})+t_q(E_{0,\alpha^3})=-\frac{3}{q}-\frac{1}{q}G_0^3\cdot 3T^0(\alpha)=0.$$

If however $k\equiv 2 \pmod{3}$, then $\alpha$ is a cube, and  (\ref{fro}) implies that
$$t_q(E_{0,\alpha})=t_q(E_{0,\alpha^2})=t_q(E_{0,\alpha^3}).$$ In particular, all curves with $j$ invariant equal to 0 will have the same trace of Frobenius.  We have already shown that $E_{24}$ is a curve with $j$ invariant equal to 0 with trace of Frobenius equal to
$$t_q(E_{24})=-q\cdot  \Fha{\rho}{\rho^2}{\epsilon}{27\cdot 24^{-1}}_q.$$
 Using this equality then gives
$$\frac{1}{3}(t_{p^{k-2}}(E_{0,\alpha})+t_{p^{k-2}}(E_{0,\alpha^2})+t_{p^{k-2}}(E_{0,\alpha^3}))=-p^{k-2}\Fha{\rho}{\rho^2}{\epsilon}{9 \cdot 8^{-1}}_{p^{k-2}}.$$
\end{proof}

\bibliographystyle{amsplain}

\begin{thebibliography}{1}
\bibitem{ahlg} S. Ahlgren, \textit{The points of a certain fivefold over finite fields and the twelfth power of the eta function}, Finite Fields Appl. \textbf{8} (2002), no. 1, 18-33.

\bibitem{ahlgo}S. Ahlgren and K. Ono, \textit{Modularity of a certain Calabi-Yau threefold}, Monatsh. Math. \textbf{129} (2000), no. 3, 177-190.

\bibitem{cox} D.A. Cox, \textit{Primes of the  Form $x^2+ny^2$. Fermat, Class Field Theory and Complex Multiplication}, A Wiley-Interscience Publication, John Wiley and Sons, New York, 1989.

\bibitem{clemens} C. H. Clemens, \textit{A Scrapbook of Complex Curve Theory}. Graduate Studies in Mathematics, vol. 55, Plenum Press, New York, 1980.

\bibitem{frech}S. Frechette, K. Ono, and M. Papanikolas, \textit{Gaussian hypergeometric functions and traces of Hecke operators}, Int. Math. Res. Not. (2004), no. 60, 3233-3262.

\bibitem{fuselier}J. Fuselier, \textit{Hypergeometric functions over $\F_p$ and relations to elliptic curves and modular forms}, Proc. Amer. Math. Soc. \textbf{138} (2010), 109-123.

%
\bibitem{greene} J. Greene, \textit{Hypergeometric functions over finite fields}, Trans. Amer. Math. Soc. \textbf{301} (1987), no. 1, 77-101.


\bibitem{hij1}H. Hijikata, A.K. Pizer, and T.R. Shemanske, \textit{The basis problem for modular forms on $\Gamma_0(N)$}, Memoirs of the American Mathematical Society \textbf{82} (1989), no. 418, vi+159.

\bibitem{hus}  D. Husem\"{o}ller, \textit{Elliptic Curves}, Graduate Texts in Mathematics, vol. 111, Springer-Verlag,  New York, 2004.


\bibitem{igusa} J. Igusa, \textit{Class number of a definite quaternion with prime discriminant}. Proc. Nat. Acad. Sci., \textbf{44} (1958), 312-314.

\bibitem{ihara} Y. Ihara, \textit{Hecke polynomials as congruence $\zeta$ functions in elliptic modular case}, Ann. of Math. \textbf{85} (1967), no. 2, 267-295.

\bibitem{ire} K. Ireland and M. Rosen, \textit{A Classical Introduction to Modern Number Theory}, 2nd ed., Graduate Texts in Mathematics, vol. 84, Springer-Verlag, New York, 1990.

\bibitem{katz} N. Katz, \textit{Exponential Sums and Differential Equations}, Annals of Mathematics Studies no. 124, Princeton University Press, New Jersey, 1990.

\bibitem{koike} M. Koike, \textit{Orthogonal matrices obtained from hypergeometric series over finite fields and elliptic curves over finite fields}, Hiroshima Math. J. \textbf{25} (1995), 43-52.

\bibitem{lang} S. Lang, \textit{Cyclotomic Fields I and II}, Graduate Texts in Mathematics, vol. 121, Springer-Verlag, New York, 1990.

\bibitem{lennon} C. Lennon, \textit{Gaussian hypergeometric evaluations of traces of Frobenius for elliptic curves}, to appear in Proceedings of the AMS.

\bibitem{manin}  J. I. Manin. \textit{The Hasse-Witt matrix of an algebraic curve}. Trans. Amer. Math.  Soc., 45:245-264, 1965.

\bibitem{miret} J. Miret, R. Moreno,  A. Rio, and M. Valls, \textit{Computing the $\ell$-power torsion of an elliptic curve over a finite field},  Math. Comp. \textbf{78} (2009), 1767-1786.

\bibitem{ono} K. Ono, \textit{Values of Gaussian Hypergeometric Series}, Trans. Amer. Math. Soc. \textbf{350} (1998), 1205-1223.

\bibitem{schoof1} R. Schoof, \textit{Counting points on elliptic curves over finite fields}, J. Theorie des Nombres de Bordeaux
\textbf{7} (1995), 219-254.

\bibitem{schoof} R. Schoof, \textit{Nonsingular plane cubic curves over finite fields}, J. Combin. Theory, Ser. A \textbf{46} (1987), no. 2, 183-211.

\bibitem{silv} J.H. Silverman, \textit{The Arithmetic of Elliptic Curves}, Graduate Texts in Mathematics, vol. 106, Springer-Verlag, New York, 1986.

\end{thebibliography}

\end{document}